\documentclass[a4paper,11pt,reqno,noindent]{amsart}
 \usepackage[centertags]{amsmath}
 \usepackage{amsfonts,amssymb,amsthm} 
 \usepackage{hyperref}

 \hypersetup{
     pdfmenubar=false,        
      pdfauthor={Neukamm, Stefan and Olbermann, Heiner},     
     pdfnewwindow=true,      
     colorlinks=false,       
     linkcolor=blue,          
     citecolor=blue,        
     filecolor=magenta,      
     urlcolor=cyan           
 }
\usepackage{graphicx} 
\usepackage{psfrag}
\usepackage[english]{babel}
\usepackage{newlfont}
\usepackage{color}
\usepackage[body={15cm,21.5cm},centering]{geometry} 
\usepackage{fancyhdr}
\usepackage{esint}
\usepackage{enumerate}
\usepackage{pstricks, pst-node}

\newtheorem{theorem}{Theorem}
\newtheorem{lemma}{Lemma}
\newtheorem{proposition}{Proposition}
\newtheorem{corollary}{Corollary}

\theoremstyle{definition}
\newtheorem{definition}{Definition}

\newtheorem{assumption}{Assumption}
\newtheorem{remark}{Remark}

\newcommand{\secf}{\boldsymbol{I\!I}}
\newcommand{\sym}{\mathrm{sym}}
\newcommand{\loc}{\mathrm{loc}}
\newcommand{\ho}{\mathrm{hom}}

\newcommand{\dist}{\operatorname{dist}}

\newcommand{\e}{\varepsilon}

\newcommand{\dto}{{\downarrow}}

\newcommand{\R}{\mathbb{R}}
\newcommand{\Z}{\mathbb{Z}}
\newcommand{\N}{\mathbb{N}}

\newcommand{\Y}{\mathcal{Y}}
\renewcommand{\d}{{d}}

\newcommand{\id}{\mathrm{Id}}

\renewcommand{\i}{\mathrm{i}}

\newcommand\wto{\rightharpoonup}

\newcommand\wtto{\stackrel{2}{\rightharpoonup}}
\newcommand\stto{\stackrel{2}{\rightarrow}}

\newcommand{\cZ}{{\chi_{\nabla u}}}
\newcommand{\ck}{{\chi_{k}}}

\newcommand{\ct}{{\chi_{\nabla u,T}}}
\newcommand{\cts}{{\chi_{\nabla u,T}^*}}

\newcommand{\step}[1]{\noindent \textbf{Step #1.}}

\def\av{\textrm{av}}
\def\hom{\textrm{hom}}
\def\iso{\textrm{iso}}

\author[S. Neukamm]{Stefan Neukamm}
\address[Stefan Neukamm]{Weierstrass Institute for Applied Analysis and
  Stochastics, Mohrenstra\ss e 39,
  D-10117 Berlin, Germany}
\email{stefan.neukamm@wias-berlin.de}
\author[H. Olbermann] {Heiner Olbermann}
\date{\today}
\address[Heiner Olbermann]{Hausdorff Center for Mathematics \& Institute for Applied Mathematics,
University of Bonn, Endenicher Allee 60, 53115 Bonn, Germany}
\email{heiner.olbermann@hcm.uni-bonn.de}

\title{Homogenization of the nonlinear bending theory for plates}

\begin{document}


\date{\today}

\xdefinecolor{green}{rgb}{0,0.4,0.1}
\renewcommand{\H}[1]{{\color{green} #1}}
\renewcommand{\S}[1]{{\color{blue} #1}}

\maketitle
\begin{abstract}
We carry out the spatially periodic homogenization of nonlinear bending theory for
plates. The derivation is rigorous in the sense of
$\Gamma$-convergence. In contrast to what one naturally would expect,
our result shows that the limiting functional is not simply a
quadratic functional of the second fundamental form of the deformed plate  as it is the case in nonlinear plate theory. It turns out that the limiting functional discriminates between whether the deformed plate is locally shaped like a ``cylinder'' or not.
For the derivation we investigate the oscillatory behavior of
sequences of second fundamental forms associated with isometric
immersions of class $W^{2,2}$, using two-scale convergence. This is a
 non-trivial task, since one has to treat two-scale convergence
in connection with a nonlinear differential constraint.
\end{abstract}

\section{Introduction}
In this article we study the periodic homogenization of the nonlinear plate model introduced by
Kirchhoff in 1850. In that model the elastic behavior of thin plates
-- undergoing bending only -- are described as follows: The reference configuration of the plate
in its undeformed, flat state is modeled by a bounded Lipschitz domain
$S\subset\R^2$, while \textit{bending deformations} are described
by \textit{isometric immersions} $u:S\to\R^3$ -- differentiable maps
that satisfy the isometry constraint
\begin{equation}\label{ass:isom}
  \partial_ju\cdot \partial_ju=\delta_{ij},
\end{equation}
where $\delta_{ij}$ denotes the Kronecker delta. The elastic \textit{bending energy} of the deformed
plate $u(S)$ is given by the variational integral 
\begin{equation}\label{eq:energy-const}
  \int_SQ(\secf),
\end{equation}
where $\secf$ is the second fundamental form associated with $u$ (see
\eqref{eq:def:secf} below), and $Q$ is
the quadratic energy density from linearized elasticity. We are
interested in the minimizers of
\eqref{eq:energy-const}, since they are related to equilibrium shapes of thin elastic plates subject to external forces and
boundary conditions. 
Indeed, Friesecke, James, M\"uller obtained in their celebrated work \cite{Friesecke-James-Mueller-02} Kirchhoff's nonlinear plate model
from nonlinear
three-dimensional elasticity in  the zero-thickness limit. The connection is rigorous in the sense
of $\Gamma$-convergence, which roughly speaking means that (almost) minimizers to a large
class of minimization problems from three-dimensional nonlinear
elasticity converge to solutions to minimization problems associated
with the bending energy \eqref{eq:energy-const}. 
\medskip

The energy density $Q$ encodes the elastic properties of the material
and, when the material is heterogeneous, depends on $x\in S$ in
addition. In the case of a periodic composite material with small period $\e\ll 1$,
the energy density might be written in the form $Q(\frac{x}{\e},F)$
where $Q(y,F)$ is periodic in $y$. For definiteness, let $Q$ satisfy the
following

\begin{assumption}\label{ass}
  Let $Q:\R^2\times\R^{2\times 2}_{\sym}\to[0,\infty)$  be
\begin{itemize}
\item[(Q1)] measurable and $[0,1)^2$-periodic in $y\in\R^2$,
  \smallskip
\item[(Q2)] convex and quadratic in $F\in\R^{2\times 2}$,
  \smallskip
\item[(Q3)] bounded and non-degenerate in the sense of
  \begin{equation}\label{ass:Q}
    \alpha |\sym F|^2\leq Q(y,F)\leq \frac{1}{\alpha}|\sym F|^2
  \end{equation}
  for all $A\in\R^{2\times 2}$, almost every $y\in\R^2$ and for some constant of ellipticity $\alpha>0$ which is fixed from now on.
\end{itemize}
\end{assumption}
We reformulate the bending energy \eqref{eq:energy-const} as the functional $\mathcal E^\e:\,L^2(\Omega,\R^3)\to[0,\infty]$ given by
\begin{equation}\label{eq:-1}
  \mathcal E^\e(u):=
  \left\{\begin{aligned}
    &\int_S Q\left(\frac{x}{\e},\secf(x)\right)\,dx&\qquad&\text{for
    }u\in W^{2,2}_\iso(S),\\
    &\infty&&\text{else,}
  \end{aligned}\right.
\end{equation}
where $W^{2,2}_\iso(S)$ denotes the subset of maps $u\in
W^{2,2}(S,\R^3)$ that satisfy \eqref{ass:isom} almost everywhere in
$S$. 
\medskip

Our goal is to understand the homogenization limit, $\e\downarrow 0$, in the
spirit of $\Gamma$-convergence. For the description of the limit we
need to classify the geometry of surfaces $u(S)$ with $u\in
W^{2,2}_\iso(S)$. 
For simplicity, let us first assume that $u$ is a smooth isometric immersion.
Since $S$ is flat,
the Gauss curvature of the surface $u(S)$ vanishes, and by a classical
result from geometry we know that locally $u(S)$ is either
\textit{flat} (when $u$ is affine), or a developable surface. In the
latter case the surface has either the shape  of a \textit{cylinder}
or a \textit{cone}. (With slight abuse of the standard terminology, we
refer to tangent developable surfaces as \textit{cones}.) For the flat
part of the surface $u(S)$ we introduce the notation
\[  
C_{\nabla u}  =  \{\,x\in S\,:\,\text{$u(S)$ is affine in a neighborhood
  of $u(x)$}\,\}.
\]
By developability, for every point $x\in S\setminus C_{\nabla u}$, there exists a unit vector $N(x)$ such that $\nabla u$ is constant on the line
segment through $x$ with direction $N(x)$.  If there exists a unit vector $\bar
N$ such that the
set $N^{-1}(\bar N)$ has density 1 at $x$, we say that the surface has
 the shape of a \textit{cylinder}  there, and we call  $x$ a
\textit{cylindrical} point. Points $x\in S\setminus
C_{\nabla u}$ where this does not hold true  will be called \textit{conical
  } points. (This dichotomy is only valid up to a null set, cf.~Definition~\ref{def:1}.) We write $Z_{\nabla u}$ and $K_{\nabla
  u}$ to denote the set of cylindrical and conical points,
respectively. As we explain in Section~\ref{sec:geometry-basic} below, the
assumption that $u$ is smooth is unnecessary, and   these notions
extend to  $W^{2,2}$-isometric immersions, see
Definition~\ref{def:1}.
\medskip

For the definition of the limiting functional we require averaged and homogenized versions of $Q$. Since the second
  fundamental form almost surely belongs to the cone of symmetric
  $2\times2$-matrices with rank at most one,
  it suffices to define the relaxed versions of $Q$ for such matrices: for a
  unit vector $T\in\R^2$ and $\mu\in\R$ set
\begin{eqnarray}
    \label{def:av-formula}
    Q_{\av}(\mu T\otimes T)&:=&\mu^2\int_{(0,1)^2}Q(y, T\otimes
    T)\,dy,\\
    \label{def:hom-formula}
    Q_{\hom}(\mu T\otimes
    T)&:=&\mu^2\min_{\alpha\in
      W^{1,2}_{T\text{-per}}(\R)}\Big\{\,\int_{(0,1)^2}Q\Big(\,y,\left(1+\alpha'(T\cdot y)\right) T\otimes
     T\,\Big)\,dy\, \Big\};
  \end{eqnarray}
  here $W^{1,2}_{T\text{-per}}(\R)$ denotes the closure w.~r.~t. the
  $W^{1,2}$-norm of the set of doubly periodic functions in
  $C^\infty(\R)$ with periods $T\cdot e_1$ and
  $T\cdot e_2$, see Subsection~\ref{sec:homog-effects} for details.
  Note that the expression for $Q_{\hom}$ differs from the usual
  formula used for the homogenization of convex integrands -- in fact,
  as we will see in Subsection~\ref{sec:homog-effects}, it can be
  interpreted as mixture of a one-dimensional averaging and homogenization.
  \medskip

    The $\Gamma(L^2)$-limit of $\mathcal E^\e$ is then given by the functional $\mathcal E_\hom:\,L^2(\Omega,\R^3)\to[0,\infty]$,
  \begin{equation*}
    \mathcal E^\hom(u):=
    \left\{\begin{aligned}
        &\int_{S}(1-\cZ(x)) Q_{\av}(\secf)+\cZ(x) Q_\hom(\secf)&\qquad&\text{for
        }u\in W^{2,2}_\iso(S),\\
        &\infty&&\text{else,}
      \end{aligned}\right.
  \end{equation*}
  where $\cZ$ denotes the indicator function of
  $Z_{\nabla u}$, see Definition~\ref{def:1} below.
  \medskip

  We shall consider boundary conditions of the following form: Let
  $L_{BC}\neq\emptyset$ denote a line segment of the form $L_{BC}=\{\,x_0+t
  N\,:\,t\in\R\,\}\cap S$ (for some $x_0\in\R^2$ and some
  unit vector
  $N\in\R^2$). We assume that
  \begin{equation}\label{ass:BC}\tag{BC}
    u=\varphi_{BC}\mbox{ and }\nabla u=\nabla\varphi_{BC}\qquad\mbox{ on }L_{BC},
  \end{equation}
  where $\varphi_{BC}:\R^2\to\R^3$ is a fixed rigid isometric immersion, i.~e.
  $\nabla\varphi_{BC}$ is constant and satisfies \eqref{ass:isom}. 
  \medskip

  We are now in position to state our main result.

  \begin{theorem}
    \label{T:1}
    Let $S\subset\R^2$ be a convex Lipschitz domain and let $Q$ satisfy (Q1) -- (Q3).
    \begin{enumerate}[(a)]
    \item Consider  $u^\e\in L^2(S,\R^3)$ with finite energy, i.~e.
      \begin{equation*}
        \limsup\limits_{\e\downarrow 0}\mathcal E^\e(u^\e)<\infty.
      \end{equation*}
      Then there exists $u\in W^{2,2}_\iso(S)$ such that $u^\e-\fint_S u^\e\to u$ in
      $L^2(S,\R^3)$ as $\e\downarrow 0$ (after possibly passing to subsequences).
    \item Let $u^\e$ converge to some $u$ in
      $L^2(S,\R^3)$ as $\e\downarrow 0$. Then
      \begin{equation*}
        \liminf\limits_{\e\downarrow 0}\mathcal E^\e(u^\e)\geq
        \mathcal E^\hom(u).
      \end{equation*}
    \item For every $u\in L^2(S,\R^3)$ there exists a sequence 
      $u^\e\in L^2(S,\R^3)$ that converges to $u$ and
      \begin{equation*}
        \lim\limits_{\e\downarrow 0}\mathcal E^\e(u^\e)=\mathcal E^\hom(u).
      \end{equation*}
      Moreover, if $u\in W^{2,2}_\iso(S)$ satisfies \eqref{ass:BC},
      then $u^\e$ can be chosen such that $u^\e\in W^{2,2}_\iso(S)$ satisfies the boundary
      condition \eqref{ass:BC} in addition.
    \end{enumerate}
  \end{theorem}
  \smallskip

  The limit $\mathcal E^\hom$ is not a standard Kirchhoff plate model. In
  particular, it is not possible to recast $\mathcal E^\hom$ into the
  form of
  \eqref{eq:energy-const}.
  Still, it is a generalized Kirchhoff plate model in the sense that
  the energy locally is quadratic in the second fundamental form.

  \begin{remark}
    \begin{itemize}
\item[(1)]
The result also holds true for non-convex Lipschitz domains $S$
    with the property that there exists some $\Sigma\subset \partial
    S$ with $\mathcal H^1(\Sigma)=0$, and the outer normal to $S$
    exists and is continuous on $\partial S\setminus\Sigma$. We limit
    ourselves to the convex case here for the sake of brevity. Our
    main point is the proof of part (b) of Theorem \ref{T:1}, which is
    completely independent of whether $S$ is convex or not. The
    construction of a recovery sequence in part (c) however becomes
    somewhat more involved for non-convex domains. It is nevertheless
    possible by appealing to the results of
    \cite{Hornung-11a} and \cite{Hornung-11b}.
\item[(2)]
We have chosen to  set the boundary conditions (\ref{ass:BC}) on a line segment in
the interior of the domain. We have done so for the sake the
simplicity. If the boundary of $S$ contains a flat part, we could also set the
boundary conditions there. It is possible to treat this case by enlarging the
domain and extending the isometric immersion affinely -- in this way, the
boundary conditions on the flat part of the initial domain become boundary
conditions on a line segment lying in the interior of the enlarged domain.
\end{itemize}
  \end{remark}
  \medskip

  Let us comment on the proof of Theorem~\ref{T:1}. Since $\mathcal
  E^\e$ is non-convex and singular with non-convex effective domain the derivation of the $\Gamma$-limit is subtle and
  standard tools, e.~g. compactness and representation results for $\Gamma$-limits that rely on integral representations, are not applicable. To overcome these difficulties we take advantage of two observations:
First, as a functional of the second fundamental form the mapping
\begin{equation}\label{eq:energy-sf}
  \secf\mapsto \int_SQ(\tfrac{x}{\e},\secf)\,dx
\end{equation}
is convex and quadratic, so that we can pass to the
limit $\e\downarrow 0$ in \eqref{eq:energy-sf} by classical homogenization
techniques, in particular two-scale convergence. Secondly, the
nonlinear isometry constraint yields a strong rigidity and allows the
second fundamental form to oscillate only in a very restricted way.
\medskip

This second observation is the heart of the matter and requires to
describe the structure of two-scale limits of vector fields under a
\textit{nonlinear} differential constraint, cf. Remark~\ref{R:MONGE}
for more details. While the interplay of two-scale convergence and
\textit{linear} differential constraints is reasonably well
understood, e.g. see \cite{Fonseca-Kroemer-10}, in the nonlinear case
no systematic approach seems to be available. In fact, to our
knowledge our result is the first attempt in that direction
in the nonlinear case. Since the main focus of this paper is the
derivation of the $\Gamma$-limit to $\mathcal E^\e$, we content ourselves with a partial identification of the two-scale limit which is yet strong enough to treat Theorem~\ref{T:1}. To motivate this in more detail consider a sequence $u^\e$ that weakly
converges in $W^{2,2}_\iso(S)$ to some limit $u$. Let $\secf^\e$ denote the second
fundamental form associated with $u^\e$. Since $\secf^\e$ is bounded
in $L^2(S,\R^{2\times 2})$, we may pass to a weakly two-scale convergent
sequence. Since $Q(y,F)$ is convex in $F$, standard results from
two-scale convergence, cf. Lemma~\ref{L:twoscale-lsc}, yield the lower bound
\begin{equation*}
  \liminf\limits_{\e\downarrow
    0}\int_SQ(\frac{x}{\e},\secf^\e(x))\,dx\geq\inf_{H(x,y)}\int_{S\times
    (0,1)^2}Q(y,H(x,y))\,dydx,
\end{equation*}
where the infimum is taken over all weak two-scale limits $H(x,y)$ of arbitrary
subsequences of $\secf^\e$. Seeking for a lower bound that only
depends on the limit $u$, we need to identify the class of limits
$H(x,y)$ that might emerge as weak two-scale limits of $\secf^\e$.
This is done in Section~\ref{S:3}. As we shall see in Proposition~\ref{P:1} only certain oscillations on
scale $\e$ are compatible with the nonlinear isometry constraint
\eqref{ass:isom}. Loosely speaking, we observe that on cylindrical
regions of the limiting plate $u(S)$, only oscillations on scale $\e$ parallel to the
line of curvature are possible, while on conical regions all
oscillations on scale $\e$ are suppressed.
\medskip

Theorem~\ref{T:1} is a homogenization result for a singular integral
functional whose effective domain $\{\mathcal E^\e<\infty\}$ is
\textit{non-convex}. Questions regarding homogenization and relaxation
of singular integral functionals related to hyperelasticity
have been actively studied in the last years, e.g. \cite{Hafsa-Mandallena-11} and the
references therein. Typically, these interesting works study integral functionals of
the form $u\mapsto \int W(\tfrac{x}{\e},\nabla u(x))\,dx$ where $u$
denotes a deformation and $W$ satisfies non-standard growth conditions
allowing for attainment of the value $+\infty$. Compared to that, in our situation the singular
behavior is of different nature. It is due to the non-convex differential
constraint \eqref{ass:isom} and, thus, requires a completely different
approach.
\medskip

The paper is organized as follows: In
Subsection~\ref{sec:homog-effects} we discuss the homogenized quadratic
form $Q_{\hom}$ in more detail. In Subsection~\ref{S:rel3d} we put our
limiting model $\mathcal E^\hom$ in relation with models derived from
three-dimensional elasticity via simultaneous dimension reduction and
homogenization. In Section~\ref{S:2} we recall some basic
preliminaries from geometry and two-scale convergence.
Section~\ref{S:3} is the core of the paper. There we analyze the
structure of oscillations of the second fundamental form. Finally, in
the last section we give the proof of Theorem~\ref{T:1}.

\subsection{Homogenization formula and homogenization
  effects}
\label{sec:homog-effects}
Theorem \ref{T:1} states in particular that locally,  there are no homogenization effects
if the deformation $u$ is not a cylindrical isometric immersion. On the cylindrical part non-trivial homogenization effects occur and the effective behavior is captured by $Q_{\ho}$ which
is defined via \eqref{def:hom-formula}. The formula involves the space
$W^{1,2}_{T\text{-per}}(\R)$ which is defined as follows: For any unit
vector $T\in\R^2$ we set 
\begin{equation*}
  C^1_{T\text{-per}}(\R):=\{\,\alpha\in C(\R)\,:\,\alpha(s+T\cdot
  k)=\alpha(s)\text{ for all }s\in\R,\,k\in\Z^2\,\}\,,
\end{equation*}
and define $W^{1,2}_{T\text{-per}}(\R)$ as the closure of $C^1_{T\text{-per}(\R)}$ w.~r.~t. the norm 
\begin{equation*}
  \|\alpha\|^2_{W^{1,2}_{T\text{-per}}(\R)}:=\int_{(0,\max\{T\cdot
  e_1,T\cdot e_2\})}\alpha^2(s)+|\alpha'(s)|^2\,ds.
\end{equation*}
The space $C^1_{T\text{-per}}(\R)$ and thus $W^{1,2}_{T\text{-per}}(\R)$ can
be characterized as follows: Consider
\begin{equation*}
  \mathcal S^1_*:=\{\,T\in \mathcal S^1\,:\,T\in r\Z^2\text{ for some }r\in\R\,\}
\end{equation*}
We will call $T\in\mathcal S^1_*$ a ``rational'' direction. We define

\begin{equation}\label{def:r}
  r(T):=\left\{\begin{aligned}
      &\sup\{\,r>0\,:\,T\in r\Z^d\,\}&&\text{if }T\in\mathcal S^1_*,\\
      & 0&&\text{otherwise}.    
  \end{aligned}\right.
\end{equation}
If the ratio of the components of $T$, i.~e. $T\cdot e_1$ and $T\cdot
e_2$, is irrational, then  $r(T)=0$ and $C^1_{T\text{-per}}(\R)$ only
contains the constant functions. Otherwise $C^1_{T\text{-per}}(\R)$
consists precisely of those functions in $C^1(\R)$ that are periodic with
period $r(T)$.
\medskip

Next, we obtain a more explicit formula for $Q_{\ho}(T\otimes T)$. If $r(T)=0$, then we have
$Q_{\ho}(T\otimes T)=Q_{\av}(T\otimes T)$. Otherwise, consider for
$t\in[0,r(T))$ the finite union of line segments
\begin{equation*}
  L_t:=\left\{y\in [0,1)^2:T\cdot y-t\in r(T)\Z\right\},
\end{equation*}
and define $q_{\av,T}:[0,r(T))\to\R$ by
\begin{equation}
  q_{\av,T}(t)=r(T)\int_{L_t}Q(y,T\otimes T)\d\mathcal H^1(y)\label{eq:31},
\end{equation}
which in fact is an average since $\mathcal H^1(L_t)=r(T)^{-1}$ for
all $t\in[0,r(T))$. With this notation, we have by Fubini
\begin{equation*}
  Q_{\ho}(T\otimes T)=\min\left\{\fint_0^{r(T)}q_{\av,T}(t)(1+\alpha'(t))^2\d
    t:\alpha\in W^{1,2}_{T\text{-per}}(\R)\right\}\,.
\end{equation*}
The solution of this one-dimensional minimization problem which is obtained by integrating the
associated Euler-Lagrange equation is well known. A minimizer
$\alpha_*$ (whose dependency on $T$ we suppress in the notation) is given
by
\begin{equation}\label{eq:hom-formula-minimizer}
  \alpha_*(t):=\frac{1}{\fint_0^{r(T)}\frac{ds}{q_{\av,_T}(s)}}\int_0^t\frac{ds}{q_{\av,T}(s)}
\end{equation}
and we obtain
\begin{equation}
  \label{eq:hom-formula-3}
  Q_{\ho}(T\otimes T)=\fint_0^{r(T)}q_{\av,T}(t)(1+\alpha'_*(t))^2 dt=\frac{1}{\fint_0^{r(T)}\frac{dt}{q_{\av,T}(t)}}\,.
\end{equation}
Thus we have averaging in the direction perpendicular to $T$ (eq.~\eqref{eq:31}) and
homogenization in the direction of $T$ (eq.~\eqref{eq:hom-formula-3}).  The averaging takes place over a  set of $\mathcal
H^1$-measure $r(T)^{-1}$, and the homogenization takes place over a set of $\mathcal
H^1$-measure $r(T)$. The better $T$
agrees with the periodic microstructure of the material (which by
assumption (Q1) is aligned with the coordinate axes), the smaller
is  $r(T)$. Hence, the better $T$ is chosen to match with the
coordinate axes, the more room there is for homogenization effects to make the
material softer with respect to bending in this direction.
\subsection{Relation to  3d nonlinear elasticity}\label{S:rel3d}
As mentioned in the introduction Kirchhoff's plate model can be
rigorously derived from nonlinear 3d elasticity. In the following we
compare the limit $\mathcal E^{\hom}$ from Theorem~\ref{T:1} to effective models
obtained from 3d elasticity via simultaneous dimension reduction and
homogenization.   To that end we consider the  energy functional
\begin{equation*}
  \mathcal  E^{\e,h}(u):=
  \frac{1}{h^2}\fint_{\Omega_h}W(\tfrac{x_1}{\e},\tfrac{x_2}{\e},\nabla
    u(x))\,dx,
\end{equation*}
where $\Omega_h:=S\times(-\frac{h}{2},\frac{h}{2})$ models the
reference domain of a thin, three-dimensional plate with
thickness $h>0$, and $W: \R^2\times\R^{3\times
  3}\to[0,\infty]$ denotes a stored energy function of an elastic
composite material. We assume that  $W(y,F)$ is $[0,1)^2$-periodic in $y$, and frame-indifferent,
non-degenerate, and $C^2$ in a neighborhood of the identity in $F$ (see
\cite{Friesecke-James-Mueller-02} for details).


\medskip

The energy $\mathcal E^{\e,h}$ models a hyperelastic material whose stress free reference state is the thin domain $\Omega_h$.
The described material is a composite that periodically  varies in
in-plane directions. Note that $\mathcal E^{\e,h}$ admits two small length scales: the
thickness $h$ and the material fine-scale $\e$. The limit $h\downarrow 0$
corresponds to dimension reduction, while $\e\downarrow 0$ amounts to
homogenization. In \cite{Friesecke-James-Mueller-02} it is shown that
$\mathcal E^{\e,h}$ $\Gamma$-converges for $h\downarrow 0$ (and fixed $\e>0$) to the energy $\mathcal E^\e$, cf. \eqref{eq:-1},
where $Q$ is obtained from the quadratic form $G\mapsto\frac{\partial^2W}{\partial F^2}(y,I)(G,G)$ by a relaxation
formula, and (by the assumptions on $W$) automatically satisfies
Assumption~\ref{ass}. Hence, in combination with Theorem~\ref{T:1} we deduce that $\mathcal
E^{\hom}$ is the double-limit of the 3d-energy $\mathcal E^{\e,h}$
that correspond to ``homogenization after dimension reduction''; i.~e.
\begin{equation*}
\mathcal E^{\hom}=\Gamma\text{-}\lim\limits_{\e\downarrow
  0}\;\Gamma\text{-}\lim\limits_{h\downarrow 0}\mathcal E^{\e,h}.
\end{equation*}
We therefore expect $\mathcal E^{\hom}$ to be a good model for the
three-dimensional plate in 
situations where $h\ll\e\ll 1$. 
\medskip

An alternative way to obtain an effective model from $\mathcal
E^{\e,h}$ is to simultaneously pass to the limit $(\e,h)\to (0,0)$. 
This has been studied in the case of rods, plates and shells, see
\cite{HorNeuVel-13, Hornung-Velcic-2012, Neukamm-10, Neukamm-12, Neukamm-Velcic-2013, Velcic-2012}. In particular, in \cite{Neukamm-12} the simpler
situation of elastic rods has been analyzed in detail, i.~e. when
$\Omega_h$ is replaced by a thin
rod-like domain of the form $(0,1)\times hB$ where $B$ denotes the
two-dimensional cross-section of the rod. As shown in \cite{Neukamm-12} the obtained
$\Gamma$-limit depends on the relative scaling between $\e$ and $h$. More precisely,
under the assumption that the ratio $\frac{h}{\e}$ converges to a
prescribed scaling factor $\gamma\in[0,+\infty]$, it is shown that the initial
energy $\Gamma$-converges to a bending torsion model for inextensible
rods, whose effective energy density continuously depends on the scaling factor $\gamma$.
Moreover, it is shown that the model obtained in the case $\gamma=0$
(which corresponds to simultaneous dimension reduction and
homogenization in the regime $h\ll \e\ll 1$) is equivalent to the
model obtained by the sequential limit ``$\e\downarrow 0$ after
$h\downarrow0$''. 
\medskip

For plates, as considered here, this suggests the following: For a given scaling
factor $\gamma>0$ consider the limit $\mathcal
E^\gamma=\Gamma\text{-}\lim_{h\downarrow 0}\mathcal E^{\e(h),h}$
where we assume that $\frac{h}{\e(h)}\to \gamma$ as $h\downarrow 0$. This limit corresponds to a simultaneous
dimension reduction and homogenization of $\mathcal E^{\e,h}$ in the
case when the fine-scale $\e$ and $h$ do not separate. The analysis for rods  described above suggests that $\mathcal E^{\hom}$
can be recovered from $\mathcal E^\gamma$ in the limit
$\gamma\downarrow 0$. Surprisingly  this is \textbf{not} the
case for plates: As shown most recently by Hornung and Vel\v{c}i\'c
and the first author in \cite[Theorem~2.4]{HorNeuVel-13}, for
$\gamma\in(0,\infty)$ the limit $\mathcal E^\gamma$ takes the form of
the plate model \eqref{eq:-1} with $Q$ replaced by the relaxed and
homogenized quadratic form $Q_\gamma$ that depends on the
scaling factor $\gamma$. A close look at the relaxation formula
defining $Q_\gamma$ shows that typically $\limsup_{\gamma\downarrow 0}Q_\gamma<Q_{\av}$.
This implies that on the level of the associated energies $\mathcal E^\gamma$ and $\mathcal E^\hom$ we typically have $\limsup_{\gamma\downarrow 0}\mathcal E^\gamma(u)<\mathcal
E^\hom(u)$ for \textit{conical} deformations $u\in
W^{2,2}_\iso(S)$, in contrast to the case of rods, where $\lim_{\gamma\downarrow 0}\mathcal E^\gamma=\mathcal E^0=\mathcal
E^\hom$.


\section*{Acknowledgements}
The authors would like to thank an anonymous referee for pointing out a mistake
in an earlier version of this manuscript, cf.~Remark
\ref{sec:sets-z_nabla-u}. This work was initiated while the first author was employed at the Max
  Planck Institute for Mathematics in the Sciences, Leipzig, Germany. The second author gratefully acknowledges the hospitality of the  Max
  Planck Institute for Mathematics in the Sciences, Leipzig, Germany.
\medskip

\section{Notation and  preliminaries}\label{S:2}
Throughout this article we use the following notation:
\begin{itemize}
\item $e_1,e_2$ denotes the standard Euclidean basis of
  $\R^2$;
\item we write $a\cdot b$ for the inner product in $\R^2$, $|\cdot|$
  for the induced Euclidean norm, and denote
  the coefficients of $a\in\R^2$ by $a_i:=a\cdot e_i$, $i=1,2$;
\item for $a=(a_1,a_2)\in\R^2$ we set $a^\perp:=(-a_2,a_1)$;
\item $\mathcal S^1:=\{\,e\in\R^2\,:\,|e|=1\}$, and $\mathcal
  S^1_*:=\{\,T\in \mathcal S^1\,:\,T\in r\Z^2\text{ for some }r\in\R\,\}$
\item for $a=(a_1,a_2),b=(b_1,b_2)\in\R^2$ we denote by $a\otimes b$
  the unique $2\times2$ matrix characterized by $e_i\cdot
  (a\otimes b)e_j=a_ib_j$;
\item we denote the entries of $A\in\R^{2\times 2}$ by $A_{ij}$ so
  that $A=\sum_{i,j=1}^2A_{ij}(e_i\otimes e_j)$, and we write $A:B:=\sum_{i,j=1}^dA_{ij}B_{ij}$ for the inner product in
  $\R^{2\times 2}$;
\item $A^t$ denotes the transposed of $A\in\R^{2\times 2}$;
\item $B(x,R)$ denotes the open ball in $\R^2$ with center $x$ and
  radius $R$;
\item $a\times b$ denotes the vector product in $\R^3$.
\end{itemize}
\subsection{Properties of $W^{2,2}$-isometric
  immersions}\label{sec:geometry-basic}
We denote by
\begin{equation*}
  W^{2,2}_{\iso}(S):=\{\,u\in W^{2,2}(S,\R^3)\,:\,u\text{ satisfies
  }\eqref{ass:isom}\text{ a.e. in }S\,\}
\end{equation*}
the set of Sobolev isometries. The second fundamental form associated with $u\in
W^{2,2}_\iso(S)$ is given by the matrix field
$\secf:S\to\R^{2\times 2}$ with entries
 \begin{equation}\label{eq:def:secf}
   \secf_{ij}:=-\partial_in\cdot\partial_j u,
 \end{equation}
where $n:=\partial_1 u\times\partial_2u$ denotes the normal
field to the surface $u(S)$.
\medskip

>From classical geometry it is well known that a smooth surface in $\R^3$
that is isometric to a flat surface is developable --- locally it is
either flat, a cylinder or a cone. As shown by Kirchheim
\cite{Kirchheim-01} (see also \cite{Pakzad-04}, \cite{Hornung-11a} and
\cite{Hornung-11b}) $W^{2,2}$-isometries share this property. In the following
we make this precise. Throughout the paper we use the notation
$[x;N]:=\{\,x+sN\,:\,s\in\R\,\}$ for the line through $x$
parallel to $N$, and $[x;N]_S$ for the connected component of $[x;N]\cap
S$ that contains $x$.
We start our survey with a regularity result on the gradient of isometries:

\begin{lemma}[see  {\cite[Proposition 5]{Mueller-Pakzad-05}}]
  \label{L:cont}
  Let $S\subset\R^2$ be a Lipschitz domain. Then $\nabla u$ is continuous for all $u\in W^{2,2}_{\iso}(S)$.
\end{lemma}

In the following let  $S$ be a convex Lipschitz
  domain and $u\in W^{2,2}_{\iso}(S)$.
We shall introduce some objects to describe the geometry of $u(S)$.
We say $x\in S$ is a \textit{flat point of $\nabla u$}, if $\nabla u$
is constant in some neighborhood of $x$ and introduce the (open) set
\begin{equation*}
  C_{\nabla u}:=\{\,x\in S\,:\,x\text{ is  a flat point of $\nabla u$}\,\}.
\end{equation*}
For our purpose it is convenient to describe the geometry of the
\textit{non-flat} part $S\setminus C_{\nabla u}$ by means of
\textit{asymptotic lines}. We say a unit vector $N\in\R^2$ is called an  \textit{asymptotic direction (for $\nabla u$) at
$x\in S$} if 
\begin{equation}\label{eq:def:asym}
  \exists s_0>0\text{ such that }\nabla u(x)=\nabla
  u(x+sN)\text{ for all $s\in(-s_0,s_0)$}.
\end{equation}
When $u$ is a smooth isometry, then it is known from classical
geometry that at every
non-flat point $x$ there exists an asymptotic direction
$N(x)$ that is unique up to a sign. In fact, we know more:  There exists a mapping
$N\,:\,S\setminus C_{\nabla u}\to \mathcal S^1:=\{\,N\in\R^2\,:\,|N|=1\,\}$ such that for all $x,y\in S\setminus
C_{\nabla u}$
\begin{subequations}
  \begin{align}
    \label{P:L:1}
    &\nabla u\text{ is constant on }[x;N(x)]_S,\\
    \label{P:L:2}
    &[x;N(x)]_S\cap[y;N(x)]_S\neq\emptyset\;\Longrightarrow\;
    [x;N(x)]=[y;N(y)].
  \end{align}
\end{subequations}
This observation extends to $W^{2,2}$-isometries:
\begin{proposition}[{\cite{Pakzad-04}}]
  \label{P:L}
  Let $u\in W^{2,2}_{\iso}(S,\R^3)$. Then there exists a locally Lipschitz
  continuous vector field 
  $N:S\setminus C_{\nabla u}\to\mathcal S^1$ such that \eqref{P:L:1}
  and \eqref{P:L:2} is true for all $x,y\in S\setminus C_{\nabla u}$.
  Furthermore, the field $S\setminus C_{\nabla u}\ni x\mapsto
  N(x)\otimes N(x)$ is unique. 
\end{proposition}
For isometries of class $C^2$, Proposition~\ref{P:L} is contained in the more general result \cite{Hartmann-Nierenberg-1959}. In the form above,
the proposition has been proven in \cite{Pakzad-04}, using ideas from
\cite{Kirchheim-01}. 

On $S\setminus C_{\nabla u}$ the second fundamental form $\secf$ is
proportional to $N^\perp\otimes N^\perp$, which has the geometric
meaning that $T:=-N^\perp$ is the principal direction along which
$u(S)$ is curved. This elementary observation is made
precise in the following lemma which can
be found in \cite{Friesecke-James-Mueller-06} and \cite{Hornung-11a}:
\begin{lemma}
  \label{L:0}
  Let $S\subset\R^2$ be bounded and $u\in W^{2,2}_\iso(S)$, then
  almost everywhere on $S$
  \begin{eqnarray}\label{L:1:1}
    \partial_i\partial_ju\cdot n&=&\secf_{ij},\\
    \label{L:1:3a}
    \partial_2\secf_{11}&=&\partial_1\secf_{12},\\
    \label{L:1:3b}
    \partial_2\secf_{21}&=&\partial_1\secf_{22},
  \end{eqnarray}
  and there exists $T:S\to\mathcal S^1$ with $T(x)=-N(x)^\bot$ for $x\in
  S\setminus C_{\nabla u}$ and $\mu\in L^2(S)$ such that
  \begin{equation}\label{L:1:2}
    \secf=\mu T\otimes T\qquad\text{ a.e.~on }S.
  \end{equation}
\end{lemma}

For a given $u\in W^{2,2}_\iso(S)$, we distinguish two subsets of the
non-flat part of $S$, which we call the \emph{cylindrical}
and the \emph{conical} part. To do so, we define  for $T\in {\mathcal S}^1$, 
\begin{align}
  \ct(x):=&\begin{cases}1 & \text{if }x\in S\setminus C_{\nabla u} \text{ and
    }N(x)\cdot T=0,\\0 & \text{else,}\end{cases}\label{eq:7}\\
  \cts(x):=&\begin{cases}\lim_{r\dto 0}\fint_{B(x,r)}\ct(y)\d y & \text{if the
      limit exists,} \\ 0 &\text{else.}\end{cases}\label{eq:6}
\end{align}

\begin{definition}
  \label{def:1}
  For $u\in W^{2,2}_{\iso}(S)$. We say $x\in S\setminus
  C_{\nabla u}$ is 
  \[
  \left\{
    \begin{aligned}
      &\text{cylindrical, and write }x\in Z_{\nabla u},&&\text{ if }\exists T\in{\mathcal S}^1:\cts(x)=1,\\
      &\text{conical, and write }x\in K_{\nabla u},&&\text{ if
      }\forall T\in {\mathcal S}^1:\cts(x)=0.
    \end{aligned}
  \right.
  \]
  We write $\cZ$ for the indicator function of $Z_{\nabla u}$.
\end{definition}

We conclude this section with some elementary properties of the introduced decomposition.

\begin{lemma}
  \label{L:Z}
  The sets $Z_{\nabla u}$ and $K_{\nabla u}$ are measurable.
  Furthermore, there exists a null set $E\subset S$ such that
  \begin{equation}
    \label{L:Z:1}
    S=C_{\nabla u}\cup Z_{\nabla u}\cup K_{\nabla u}\cup E.
  \end{equation}
  There exists a countable  set $\mathcal S_{\nabla
    u}\subset\mathcal S^1$ of pairwise non-parallel vectors, such that
  \begin{eqnarray}
    \label{L:Z:2}
    Z_{\nabla u}&=&\bigcup_{T\in\mathcal S_{\nabla u}}\{\cts=1\},\\
    \label{L:Z:3}
    \cZ&=&\sum_{T\in\mathcal S_{\nabla u}}\ct\qquad\text{a.e.~in }S.
  \end{eqnarray}
\end{lemma}
\begin{proof}
  Consider the set
  \begin{equation*}
    \mathcal T:=\{\,T\in\mathcal S^1\,:\,\mathcal L^2(\{x\in
    S\setminus C_{\nabla u}\,:\,N(x)=T^\perp\})>0\},
  \end{equation*}
  which can be written as
  \begin{equation*}
    \mathcal T=\bigcup_{k\in\N}\mathcal T_k,\qquad \mathcal T_k:=\{\,T\in\mathcal S^1\,:\,\mathcal L^2(\{x\in
    S\setminus C_{\nabla u}\,:\,N(x)=T^\perp\})>\tfrac{1}{k}\}.
  \end{equation*}
  Since $\mathcal L^2(S)<\infty$, and since the sets $\{x\in
  S\setminus C_{\nabla u}\,:\,N(x)=T^\perp\}$, $T\in\mathcal S^1$, are
  pairwise disjoint, each $\mathcal T_k$ only contains a
  finite number of elements, and thus $\mathcal T$ is
  countable. From the definition of $\cts$ it is clear that each
  $T\in\mathcal S^1$ with $\cts(x)=1$ for some $x\in S$ must be an
  element of $\mathcal T$ or $-\mathcal T$. Hence, the set
  \begin{equation*}
    \tilde{\mathcal S}:=\{\,T\in\mathcal S^1\,:\,\exists x\in S\text{
      s.t. }\cts(x)=1\,\}
  \end{equation*}
  is at most countable, and we get
  \begin{equation}
    \label{eq:11}
    Z_{\nabla u}=\bigcup_{T\in\tilde{\mathcal
    S}}\{\cts=1\}.
  \end{equation}
  Since this is a countable union of measurable sets, we deduce that
  $Z_{\nabla u}$ is measurable.   By virtue of the invariance property $\cts=\chi_{\nabla
  u,-T}^*$, we may replace in \eqref{eq:11}  the set $\tilde{\mathcal T}$ by
a suitable set $\mathcal S_{\nabla u}\subset\tilde{\mathcal T}$ of
mutually non-parallel vectors. This proves \eqref{L:Z:2}. 

By the Lebesgue Differentiation Theorem, we have $\ct=\cts$
almost everywhere in $S$. Hence, since $\mathcal S_{\nabla u}$ is countable, we can find a
common null set $A\subset S$ such that $\ct(x)=\cts(x)$ for all $x\in
S\setminus A$ and all $T\in\mathcal S_{\nabla u}$, and thus
\eqref{L:Z:3} follows.

We finally prove \eqref{L:Z:1}. Set 
\begin{equation*}
  E:=S\setminus(C_{\nabla u}\cup Z_{\nabla u}\cup K_{\nabla u}),
\end{equation*}
and let $x\in E$. Then there exists $T\in\mathcal S^1$ such that
$0<\cts(x)<1$. By the same reasoning as above, we deduce that
$T\in\mathcal T$. Since this is true for any $x\in E$, we get
$E\subset\bigcup_{T\in\mathcal T}\{0<\cts<1\}$. Since indicator
functions $\{0,1\}$-valued, the Lebesgue Differentiation Theorem
implies that $\{0<\cts<1\}$ is a null set, and thus $E$ is contained in
a countable union of null sets, and thus a null set itself.
\end{proof}

\begin{remark}
  \label{sec:sets-z_nabla-u}
We are grateful to an anonymous referee, who pointed out to us that neither
$Z_{\nabla u}$ nor $K_{\nabla u}$ can be sensibly defined as open sets. Indeed,
these sets (defined as above) could be of positive measure, but not contain any open ball -- i.e.,
they could be of fat Cantor type.
\end{remark}

\medskip

\subsection{Two-scale convergence.}\label{sec:2.1}
Let $Y=[0,1)^2$ denote the unit cell in $\R^2$, and let
$\Y:=\R^2/\Z^2$ denote the unit torus. We denote by $C(\mathcal
Y)$ (resp. $C^\infty(\mathcal Y)$) the space of continuous (resp. smooth)
functions on the torus. We tacitly identify functions
in  $C(\mathcal Y)$ (resp. $C^\infty(\mathcal Y)$) with continuous
(resp. smooth), $Y$-periodic, functions on $\R^2$. We denote by
$L^2(\mathcal Y)$ (resp. $W^{1,2}(\mathcal Y)$) the closure of $C^\infty(\mathcal Y)$ as a subspace
of $L^2_\loc(\R^2)$ (resp. $W^{1,2}_\loc(\R^2)$). Note that $L^2(Y)\simeq L^2(\mathcal Y)$, while
$W^{1,2}(\mathcal Y)\neq W^{1,2}(Y)$. From \cite{Nguetseng-89} and \cite{Allaire-92} we cite the definition of weak two-scale convergence:
\begin{definition}
\label{twodef}
A bounded sequence $w^\e\in L^2(S)$  \textit{weakly two-scale converges} to
$w\in L^2(S\times\Y)$ if and only if 
\begin{equation*}
  \lim_{\e\downarrow 0}\int_S w^{\e}(x)\psi(x,x/\e)dx
=\int_{S\times Y}w(x,y)\psi(x,y)\,dx\, dy\quad \forall \psi\in C^\infty_0(S\times\Y).
\end{equation*}
Then we write $w^\e\wtto w$ in
$L^2(S\times\Y)$. If the sequence satisfies in addition
\begin{equation*}
  \lim_{\e\downarrow 0}\int_S |w^\e(x)|^2dx
=\int_{S\times Y}|w(x,y)|^2\,dx\, dy
\end{equation*}
then we say that $w^\e$ is \textit{strongly two-scale convergent} to $w$ and write
$w^\e\stto w$. For vector valued functions we define weak and strong
two-scale convergence component-wise.
\end{definition}

The following result can be found in \cite{Allaire-92}. It is an elementary but
fundamental property of two-scale convergence and allows to pass to the
limit in products of weakly convergent sequences.
\begin{lemma}
  \label{L:2scale:1}
  Let $S\subset\R^2$ be open and bounded. Consider two sequences
  $w^\e$ and $\psi^\e$ that are bounded in $L^2(S)$, and suppose that
$w^\e\stto w$ strongly two-scale and $\psi^\e\wtto\psi$ weakly two-scale
in $L^2(S\times\mathcal Y)$. Then
\begin{equation*}
  \int_S w^\e(x)\psi^\e(x)\,dx\,\to\,\int_{S\times Y}w(x,y)\psi(x,y)\,dydx.
\end{equation*}
\end{lemma}

The following lemma can be found as Proposition 2.12 in
\cite{Visintin-06} and is helpful for the computation of strong two-scale limits
for products.
\begin{lemma}
\label{L:stto}
Let $p,q\geq 1$, and let $v^\e$, $w^\e$ be sequences in $L^p(S)$, $L^q(\Y)$ respectively, with $v^\e\to v$ in $L^p(S)$ and $w^\e\to w$ in $L^q(\Y)$. Then
\[
v^\e(x)w^\e(x/\e)\stto v(x)w(y)\quad \text{ in  }L^r(S\times\Y)\,,
\]
where $r^{-1}=p^{-1}+q^{-1}$.
\end{lemma}

Two-scale convergence allows to conveniently pass to limits in convex
functionals with periodic coefficients. The following lemma is a
special case of \cite[Proposition~1.3]{Visintin-07} 
\begin{lemma}
  \label{L:twoscale-lsc}
  Let $A\subset\R^2$ be open and bounded, and let $Q$ satisfy  Assumption~\ref{ass}.
  \begin{enumerate}[(a)]
  \item Suppose that $G^\e\in L^2(A,\R^{2\times 2})$ weakly two-scale
    converges to $G\in L^2(A\times\mathcal Y,\R^{2\times 2})$. Then
    \begin{equation*}
      \liminf\limits_{\e\downarrow 0}\int_A
      Q(\tfrac{x}{\e},G^\e(x))\,dx\geq \int_{A\times Y}Q(y,G(x,y))\,dydx.
    \end{equation*}
  \item Suppose that $G^\e\in L^2(A,\R^{2\times 2})$ strongly two-scale
    converges to $G\in L^2(A\times\mathcal Y,\R^{2\times 2})$. Then
    \begin{equation*}
      \lim\limits_{\e\downarrow 0}\int_A
      Q(\tfrac{x}{\e},G^\e(x))\,dx=\int_{A\times Y}Q(y,G(x,y))\,dydx.
    \end{equation*}
  \end{enumerate}
\end{lemma}

\section{Two-scale limits of second fundamental forms}\label{S:3}
In this section we analyze the structure of two-scale limits
of second fundamental forms. We consider the following generic situation:
\begin{itemize}
\item[(LB)] Let $u^\e$ be a sequence in $W^{2,2}_{\iso}(S)$, let $u\in
  W^{2,2}_{\iso}(S)$, and let $G\in L^2(S\times\mathcal Y,\R^{2\times
    2})$. Suppose that
  \begin{equation}\label{P:1:ass3}
    \left\{\begin{aligned}
      u^\e\wto\,& u&&\text{weakly in }W^{2,2}(S),\\
      \secf^\e\wtto\,&\secf(x)+G(x,y)&&\text{weakly two-scale in
      }L^2(S\times\mathcal Y,\R^{2\times 2}),
    \end{aligned}\right.
  \end{equation}
  as $\e\downarrow 0$.
\end{itemize}
(Note that (LB) is generic, since from  every sequence $u^\e\in W^{2,2}_{\iso}(S)$ that is
bounded in $W^{2,2}(S)$ we may extract a subsequence that
satisfies (LB)). The two-scale field $G$ captures certain modes of oscillations of $\secf^\e$ that emerge in the limit $\e\downarrow
0$. Our goal is to understand and identify the structure of $G$.  
\medskip

Some information on $G$ can easily be obtained by standard results of
two-scale convergence:   As a consequence of \eqref{L:1:3a} and \eqref{L:1:3b} we may
represent the second fundamental form of an arbitrary isometry as the
Hessian of a scalar field. In particular,  $\secf^\e=\nabla^2\varphi^\e$ for some
$\varphi^\e\in W^{2,2}(S)$. As an immediate consequence, we find that $G(x,y)=\nabla^2_y\psi(x,y)$ where $\psi\in
L^2(S,H^2(\mathcal Y))$. However, this simple reasoning, which does
not exploit the nonlinear constraint \eqref{ass:isom}, is far from being optimal. In fact,
below we show that
oscillations of $\secf^\e$ on scale $\e$ are suppressed in regions where the
limiting isometric immersion $u$ is neither cylindrical nor flat.
Moreover, we prove that at points where $u$ is cylindrical,
oscillations on scale $\e$ can only emerge perpendicular to asymptotic
directions. 
\medskip

Our findings are summarized in the upcoming result, which is the main
tool in proving the lower bound for the $\Gamma$-convergence result.

\begin{proposition}\label{P:1}
  Suppose (LB). 
  Then  the following properties hold:
  \begin{enumerate}[(a)]
  \item (conical case). $G=0$ almost everywhere in $K_{\nabla
      u}\times\mathcal Y$.
    \smallskip
  \item (cylindrical case). Let $\mathcal S_{\nabla u}$ denote the set
    introduced in Lemma~\ref{L:Z}. Then for each $T\in\mathcal S_{\nabla u}\cap\mathcal
    S^1_*$ there exists a function $\alpha_T\in
    L^2(S,W^{1,2}_{T\text{-per}}(\R))$ such that
    \begin{equation}\label{eq:25}
      \begin{split}
        \cZ(x)G(x,y)\,=\,\sum_{T\in\mathcal S_{\nabla u}\cap\mathcal S^1_*}\chi_{\nabla
          u,T}(x)\partial_s\alpha_{T}(x,T\cdot y)\Big(T\otimes
        T\Big)\\\text{for a.e. }(x,y)\in S\times\mathcal Y.
      \end{split}
    \end{equation}
    Here $\partial_s\alpha_T$ denotes the derivative of $\alpha_T$
    w.~r.~t. its second component.
  \end{enumerate}
\end{proposition}
(The proof is postponed to the end of this section.)


\begin{remark}
  \label{sec:two-scale-limits}
  \begin{enumerate}
  \item In the proof of Theorem~\ref{T:1} the preceding proposition is used to establish the lower-bound
    part of the $\Gamma$-convergence statement. The proposition yields
    a characterization of the possible two-scale limits of $\secf^\e$.
    The characterization  on non-flat regions of $u$ is optimal. Yet,
    regarding flat regions,  Proposition~\ref{P:1} is partial, since
    there it does not
    yield any detailed information on the behavior of $G(x,y)$. Still, Proposition~\ref{P:1} is sufficient for identifying
    the $\Gamma$-limit in the proof of Theorem~\ref{T:1}.
  \item We would like to emphasize that on the right-hand side of
    \eqref{eq:25} only directions $T\in\mathcal S^1$ in rational
    directions appear. In particular, \eqref{eq:25} says that on the (possibly non-negligible)
    set
    \begin{equation*}
      \Big\{x\in S\,:\,{\sum_{T\in \mathcal S_{\nabla
            u}\setminus {\mathcal S}^1_*}\chi_{\nabla u,T}(x)}=1\Big\}\subset Z_{\nabla u}
    \end{equation*}
    the two-scale field $G$ vanishes. This  effect is due to the
    nature of two-scale
    convergence, which ``resolves'' only oscillations in rational
    directions and ``filters out'' oscillations in irrational
    directions. Let us remark that this behavior is beneficial for our
    purpose: Since the considered material is periodic, only oscillations adapted to
    the material's periodicity account for homogenization.
  \end{enumerate}
\end{remark}


The crucial observation in the argument of Proposition~\ref{P:1} is that in the situation of (LB), the possible
oscillations of $\secf^\e$ on the length scale $\e$ are restricted to
a very particular set, namely those parts of the domain where the asymptotic directions of the limit $u$ agree with the direction
of the oscillation. The following lemma expresses this fact on the
level of $G$. 
\color{black}

\newcommand{\ckt}{{\tilde\chi_k}}

\newcommand{\Ak}{{A_k}}


\newcommand{\Aked}{{A_k^{\e,\delta}}}
\newcommand{\cked}{{\chi_k^{\e,\delta}}}
\newcommand{\cknd}{{\chi_k^{0,\delta}}}
\newcommand{\Aknn}{{A_k^{0,0}}}
\newcommand{\cknn}{{\chi_k^{0,0}}}
\begin{lemma}
  \label{L:main}
  Suppose (LB), and let $N:S\setminus C_{\nabla u}\to\mathcal S^1$
  denote the Lipschitz field associated with $u$ via Proposition~\ref{P:L}. Then for
  every $k\in \Z^2\setminus\{0\}$ the function $f_k:S\to \R$ defined by
  \begin{eqnarray*}
    f_k(x)&:=& (1-\ckt(x))\int_Y G(x,y)\exp(2\pi\i k\cdot y)\d y,\\
    \ckt(x)&:=&
    \begin{cases}
      1 & \text{if }x\in C_{\nabla u}\text{ or if }x\in S\setminus C_{\nabla u} \text{ and
      }N(x)\cdot k=0\\0 & \text{else}
    \end{cases}
  \end{eqnarray*}
  is identically $0$ almost everywhere.
\end{lemma}
(The proof is postponed to the end of this section.)
\medskip

The argument of this result makes use of several auxiliary lemmas,
that we state next. First, we need to extend the field $N$ of
asymptotic directions, see Proposition~\ref{P:L}, to the flat
region. We only require a local extension to balls away from the
boundary of $S$. This is the content of the upcoming Lemma
\ref{L:ext}, which -- despite being elementary -- plays a crucial role in our analysis.

\begin{lemma}
  \label{L:ext}
  Let $u\in W^{2,2}_{\iso}(S)$. Consider a ball $B$ with $2B\subset S$. Then there exists a Lipschitz continuous function
  $N:B\to\mathcal S^1$ such that for all $x,y\in B$:
  \begin{align}
    \label{L:L:1}
    &\nabla u\text{ is constant on }[x;N(x)]_B,\\
    \label{L:L:2}
    &[x;N(x)]_{2B}\cap[y;N(x)]_{2B}\neq\emptyset\;\Longrightarrow\;[x;N(x)]=[y;N(y)]
  \end{align}
  Moreover, we have
  \begin{equation}\label{L:ext:lip}
    \operatorname{Lip}(N)\leq\frac{1}{\operatorname{radius}(B)}.
  \end{equation}
\end{lemma}
(The proof of Lemma~\ref{L:ext} is postponed to the end of this section.)
\medskip

Since the Lipschitz bound \eqref{L:ext:lip} only depends on the radius
of $B$, and in particular not on the isometry $u$, we get the following compactness result:

\begin{corollary}
  \label{C:1}
  Let $B$ denote a ball with $2B\subset S$. Consider a sequence
  $u^\e\in W^{2,2}_{\iso}(S)$ and let $N^\e:B\to\mathcal S^1$ denote the
  Lipschitz function associated with $u^\e$ via
  Lemma~\ref{L:ext}. Then there exists a Lipschitz function $\tilde
  N:B\to\mathcal S^1$ and $\tilde \mu\in L^2(B\times \mathcal Y)$ such that (up
  to subsequences)
  \begin{align}
    \label{C:1:1}
    N^\e\otimes N^\e&\,\to\, \tilde N\otimes \tilde N\qquad\text{uniformly
      in $B$,}\\
    \label{C:1:2}
    \secf^\e&\,\wtto\, \tilde\mu(x,y)\left(\,\tilde N^\perp(x)\otimes
      \tilde N^\perp(x)\,\right)\qquad\text{two-scale in $L^2(B\times
      \mathcal Y)$}.
  \end{align}
  Moreover, if $u^\e\wto
  u$ weakly in $W^{2,2}(S,\R^3)$ and $N:B\to\mathcal S^1$ is
  associated with $u$ via Lemma~\ref{L:ext}, then we have
  \begin{equation}
    \label{C:1:3}
    \tilde N\otimes \tilde N=N\otimes N\qquad\text{in $B\setminus C_{\nabla u}$}.
  \end{equation}
\end{corollary}
(The proof of Corollary~\ref{C:1} is postponed to the end of this section.)
\medskip

The following is a standard construction of the so-called line of curvature
coordinates;  see e.~g.
  \cite{Hornung-11a} and \cite{Hornung-11b}.
\begin{lemma}
\label{L:loc}
Let $u\in W^{2,2}_\iso(S)$ and let $\secf$ denote its second
fundamental form. Let $B=B(x_0,R)$ denote a ball with $2B\subset S$,
and denote by $N:B\to{\mathcal S}^1$ the Lipschitz field associated
with $u$ according to Lemma~\ref{L:ext}.
\begin{itemize}
\item[(i)] 
  There exists a  function  \mbox{$\Gamma\in W^{2,\infty}([-R,R], B)$} with
  \begin{equation*}
    \Gamma(0)=x_0,\qquad
    (\Gamma)'(t)=-(N^\e)^\perp(\Gamma(t))\qquad\text{for all }t\in[-R,R], 
  \end{equation*}
  and additionally
  \begin{equation}\label{L:10:st1}
    \max_{t\in[-R,R]}|\kappa(t)|\leq\frac{1}{R}
  \end{equation}
where $\kappa(t):=\Gamma''(t)\cdot N(\Gamma(t))$.
  \item[(ii)]
  For $(t,s)\in Q:=(-\frac{R}{2},\frac{R}{2})^2$ define
  $\Phi(t,s):=\Gamma(t)+sN(\Gamma(t))$. Then the map
  $\Phi: Q\to\Phi(Q)$ is
  one-to-one and Lipschitz continuous with
  \begin{equation}\label{L:10:st2:1}
    \operatorname{Lip}(\Phi)\leq 2,\qquad \frac{1}{2}\leq \det\nabla\Phi\leq 2,
  \end{equation}
  and satisfies
  \begin{equation}\label{L:10:st2:2}
    \tfrac{1}{4}B\subset \Phi(Q)\subset S.
  \end{equation}
  Moreover,  there exists $\kappa_n\in L^2((-\frac{R}{2},-\frac{R}{2}))$ such that
  \begin{eqnarray}
    \label{L:5:1}
    \secf(\Phi(t,s))&=&\frac{\kappa_n(t)}{1-s\kappa(t)}\,\Gamma'(t)\otimes\Gamma'(t),\\
    \label{L:5:2}
    \secf(\Phi(t,s))|\det\nabla\Phi(t,s)|&=&\kappa_n(t)\,\Gamma'(t)\otimes \Gamma'(t),
      \end{eqnarray}
  almost everywhere in $Q$.
\end{itemize}

\end{lemma}

(The proof of Lemma~\ref{L:loc} is postponed to the end of this
section.)
\medskip

After these preparations, we can start with the proofs of
Lemma~\ref{L:main} and Proposition \ref{P:1} in earnest. 

\begin{proof}[Proof of Lemma~\ref{L:main}]
Let $\tilde B$ be a ball such that $2 \tilde B\subset S$.
We will show that 
\begin{equation}
  \label{eq:20}
  f_k=0\qquad\text{a.e.~in }\tilde B.
\end{equation}
Since
$S$ can be covered by countably many of such balls, this proves the
claim of the lemma. 

We denote by $N^\e:\tilde B\to\mathcal S^1$ the Lipschitz function  associated
with $u^\e$ according to Lemma~\ref{L:ext}. Thanks to
Corollary~\ref{C:1} we may assume (by possibly passing to a
subsequence) that there exists a Lipschitz field $N^0:\tilde
B\to\mathcal S^1$ such that $N^\e\otimes N^\e\to N^0\otimes N^0$
uniformly in  $\tilde B$ as $\e\downarrow
0$. 
\medskip

\step 1 Decomposition of the domain.
\smallskip

For $\e\geq 0$ and $\delta> 0$ define the sets
\begin{eqnarray*}
  \Aked&:=&\{\,x\in {\tilde B}\,:\,|N^\e(x)\cdot k|< \delta\,\},\\
  \Aknn&:=&\{\,x\in {\tilde B}\,:\,N^0(x)\cdot k=0\}.
\end{eqnarray*}
We write $\cked$ for the  characteristic
function associated to $\Aked$. 
Note that 
\begin{equation}
  \label{eq:5}
  \begin{split}
    \cked\to & \cknd \text{ pointwise  as } \e\dto 0\, ,\\
    \cknd\to &\cknn\text{ pointwise as }\delta\dto 0\,.
  \end{split}
\end{equation}
The former is just a consequence of the uniform convergence
$N^\e\otimes N^\e\to N^0\otimes N^0$, and
the latter is obvious from the definitions.

Recall that $N:S\setminus C_{\nabla u}\to\mathcal S^1$
denotes the vector field associated with $u$ via
Proposition~\ref{P:L}. By \eqref{C:1:3} we have $N^0||N$ on  $\tilde B\setminus C_{\nabla
  u}$. Hence, in the definition of
$\ckt$ we may replace $N$ by $N^0$, so that $\Aknn\subset  \tilde
B\cap \{\ckt=1\}$. Consequently,  for \eqref{eq:20}, it suffices to show that the
  function
  \begin{equation*}
    \tilde f_k(x):=(1-\cknn(x))\int_Y G(x,y)\exp(2\pi\i k\cdot
    y)\d y
  \end{equation*}
is identically $0$ almost everywhere in $\tilde B$.
To show the latter, it is enough to prove $\int_{B}\tilde f_k(x)\d
x=0$ for every ball $B$ satisfying $4B\subset \tilde B$, since $\tilde B$
can be finely covered by  such
balls.\\
 From now on let $B$ be such a ball. As a consequence of \eqref{eq:5} and Lemma \ref{L:stto}, we have 
\begin{equation}
  \label{eq:3}
  \cked (x) \exp\Big(\frac{2\pi\i k\cdot x}{\e}\Big)\stto \cknd(x)\exp(2\pi\i k\cdot y) \quad \text{ in }L^2(B\times \Y)\quad\text{
    as }\e\dto 0\,.
\end{equation}
In combination with Lemma \ref{L:2scale:1}, 
and since $\int_Y\exp(2\pi \i k\cdot y)\,dy=0$ we get
\begin{eqnarray*}
  &&\lim_{\e\dto 0}\int_B\cked(x)\secf^\e(x)\exp\Big(\frac{2\pi\i k\cdot x}{\e}\Big)\d x\\
  &=&
  \int_{B\times Y}\cknd(x)(\secf(x)+G(x,y))\exp(2\pi\i k\cdot y)\d x \d y\\
  &=&
  \int_{B\times Y}\cknd(x)G(x,y)\exp(2\pi\i k\cdot y)\d x \d y.
\end{eqnarray*}

Also, for any function $f\in L^1(S)$, we have by the continuity of the integral
\[
\cknd f\to \cknn f\quad  \text{ in }L^1(S) \quad\text{ as }\delta\dto 0\,.
\]
Hence,
\begin{equation}
  \begin{split}
    \int_B\tilde f_k=&\int_{B\times Y}(1-\cknn(x))G(x,y)\exp(2\pi \i k\cdot y)\d x\d y\\
    =&\lim_{\delta\dto 0}\lim_{\e\dto
      0}\int_{B}(1-\cked(x))\secf^\e(x)\exp\Big(\frac{2\pi\i k\cdot x}{\e}\Big)\d
    x\,.\label{eq:4}
  \end{split}
\end{equation}

\step 2 Conclusion.

In view of Step~1, in order to conclude the proof we only need to
prove: for any $\delta>0$ we have
\begin{equation}
\label{eq:maineq}
\lim_{\e\dto 0}\int_B (1-\cked(x) )\secf^\e(x)\exp\Big(\frac{2\pi\i k\cdot x}{\e}\Big)\d x=0\,.
\end{equation}
In the argument we make use of the line of
curvature coordinates: An application of  Lemma \ref{L:loc} to  $u^\e$ yields a chart 
\begin{equation*}
  Q:=(-2R,2R),\qquad 
  \Phi^\e: Q\to S,\qquad
  \Phi^\e(t,s):=\Gamma^\e(t)+sN^\e(\Gamma(t))
\end{equation*}
such that $B\subset \Phi^\e(Q)\subset S$. For brevity we set $N^\e(t)=N^\e(\Gamma^\e(t))$, $T^\e(t)=-N^\e(t)^\bot$, and write
\begin{eqnarray*}
  \chi_{B}^\e(t,s):=
  \begin{cases}
    1&\Phi^\e(t,s)\in B,\\
    0&\text{else,}
  \end{cases}\qquad\text{and}\qquad \rho^{\e,\delta}(t,s):=1-\cked(\Phi^\e(t,s))
\end{eqnarray*}
 for the indicator functions of $B$ and the complement of $\Aked$ in the new
 coordinates. With this notation the associated change of coordinates reads
\begin{multline*}
  \int_{B}(1-\cked )\secf^{\e}(x)\exp\Big(\frac{2\pi\i k\cdot x}{\e}\Big)\d x\\
  =\int_{Q}\chi_B^\e(t,s)\rho^{\e,\delta}(t,s)\secf^\e(\Phi^\e(s,t))\exp\Big(\frac{2\pi\i
    k\cdot \Phi^\e(t,s)}{\e}\Big)|\det\nabla\Phi^\e(s,t)|\d s\d t.
\end{multline*}
Using the definition of $\Phi^\e$ and \eqref{L:5:2} the right-hand
side simplifies to
\[
\begin{split}
  \int_{Q}\chi_B^\e(t,s)\rho^{\e,\delta}(t,s)\kappa_n^\e(t)T^\e(t)\otimes T^\e(t) \exp\left(\frac{2\pi i
      k\cdot\Gamma^\e(t)}{\e}\right)\exp\left(s\frac{2\pi i k\cdot N^\e(t)
      }{\e}\right)\d s\d t\,.
\end{split}
\]
Since the field of asymptotic directions $N^\e$  only
depends on $t$ (in the new coordinates), it follows from the definition of $\Aked$ that
$\rho^{\e,\delta}(t,s)=\rho^{\e,\delta}(t)$ does not depend on $s$.
Hence, we get
\begin{equation*}
  \int_{B}(1-\cked )\secf^{\e}(x)\exp\Big(\frac{2\pi\i k\cdot x}{\e}\Big)\d x
  =\int_{Q}\chi_B^\e(t,s)f^\e(t)\partial_sG^\e(t,s)\d s\d t,
\end{equation*}
where
\[
\begin{split}
  f_\e(t,s)=& \kappa_n^\e(t)T^\e(t)\otimes T^\e(t)\exp\left(\frac{2\pi i
      k\cdot\Gamma^\e(t)}{\e}\right),\\
  G_\e(s,t) =&\rho^{\e,\delta}(t)\frac{\e}{2\pi\i N^\e(t)\cdot k}\exp\left(s\frac{2\pi i k\cdot N^\e
      (t)}{\e}\right).
\end{split}
\]
Note that $G^\e$ is well-defined, since $|N^\e\cdot k|^{-1}\leq
\delta^{-1}$, whenever $\rho^{\e,\delta}$ is non-zero.
Clearly, for \eqref{eq:maineq} it suffices to prove
\begin{equation}
  \lim\limits_{\e\downarrow 0}\int_{Q}\chi_B^\e(t,s)f^\e(t)\partial_sG^\e(t,s)\d s\d t=0.\label{eq:29b}
\end{equation}
To that end, we first claim that for all $t\in(-2R,2R)$:
\begin{equation}\label{eq:9}
  \left|\int_{-2R}^{2R}\chi_B^\e(t,s)\partial_sG^\e(t,s)\,ds\right|\leq 4\frac{\e}{\delta}.
\end{equation}
Indeed, since $B$ is convex, and
$s\mapsto\Phi^\e(t,s)$ is linear, we deduce that
$s\mapsto\chi^\e_B(t,s)$ is the indicator function of an open
(possibly empty) interval, say $(s_1^\e(t),s_2^\e(t))\subset
(-2R,2R)$. Hence, an integration by parts yields
\begin{equation*}
  \left|\int_{-2R}^{2R}\chi_B^\e(t,s)\partial_sG^\e(t,s)\,ds\right|
  =
  \left|\int_{s_1^\e(t)}^{s_2^\e(t)}\partial_sG^\e(t,s)\,ds\right|
  \leq 2\|G^\e\|_{L^\infty(Q)}\leq 4\frac{\e}{\delta},
\end{equation*}
which proves \eqref{eq:9}.
By Fubini's theorem and the triangle inequality, we have
\begin{equation*}
|\int_{Q}\chi_B^\e(t,s)f^\e(t)\partial_sG^\e(t,s)\d s\d t|\leq\int_{-2R}^{2R}|f^\e(t)|\left|\int_{-2R}^{2R}\chi_B^\e(t,s)\partial_sG^\e(t,s)\,ds\right|\,\d t.
\end{equation*}
To complete the proof it remains to argue
that $\int_{-2R}^{2R}|f^\e(t)|\,\d t$ is uniformly bounded  in $\e$. Here comes the argument:
\begin{eqnarray*}
  \int_{-2R}^{2R}|f^\e(t)|\,\d t&=&\frac{1}{4R}\int_{Q}|\kappa_n^\e(t)T^\e(t)\otimes T^\e(t)|\,\d t\d
  s\\
  &\stackrel{\eqref{L:5:2}}{=}&
  \frac{1}{4R}\int_{Q}|\secf^\e(\Phi^\e(t,s))||\det\nabla\Phi^\e(t,s)|\,\d t\d
s\\
  &=&
  \frac{1}{4R}\int_{\Phi^{\e}(Q)}|\secf^\e(x)|\d x\stackrel{\Phi^\e(Q)\subset
    S}{\leq}\frac{1}{4R}\int_{S}|\secf^\e(x)|\d x.
\end{eqnarray*}
Since $\secf^\e$ weakly converges in $L^2(S)$ as $\e\downarrow 0$, we
deduce that the right-hand side is uniformly bounded in $\e$.

\end{proof}

\begin{remark}\label{R:MONGE}
  As a consequence of \eqref{L:1:3a} and \eqref{L:1:3b} we may
  represent the second fundamental form $\secf$ of an arbitrary 
  $W^{2,2}$-isometry as $\secf=\nabla^2\varphi$ where $\varphi\in
  W^{2,2}(S)$ is a scalar function that solves the degenerate Monge-Amp\`ere equation
  \begin{equation}\label{eq:monge}
    \det\nabla^2\varphi=0.
  \end{equation}
  Above $\nabla^2\varphi$ denotes the Hessian of $\varphi$. As in \cite{Pakzad-04}, Proposition~\ref{P:L}
  can be reformulated for scalar functions that belong to the non-convex space
  \begin{equation*}
\mathcal A:=\{\,\varphi\in W^{2,2}(S)\,:\,\det\nabla^2\varphi=0\,\}.
  \end{equation*}
  Without much effort we recover the result of Lemma~\ref{L:main} on the level of
  the functions $\varphi\in\mathcal A$; i.~e. the following statement:
  Consider a sequence $\varphi^\e\in W^{2,2}(S)$ of solutions to
  \eqref{eq:monge} and assume that $\varphi^\e$ weakly converges to
  some $\varphi$ in $W^{2,2}(S)$, and $\nabla^2\varphi^\e$ converges weakly
  two-scale to $\nabla^2\varphi+G$ in $L^2(S\times \mathcal Y)$.  If  the limit $\varphi$ is locally  equal to an affine
  function, i.~e. for some open set $O\ni x_0$, $A\in\R^2$ and $a\in\R$ we have
  \begin{equation*}
    \int_O|\varphi(x)-(A\cdot x+a)|^2\,dx>0\,,
  \end{equation*}
we write $x_0\in C_{\nabla\varphi}$. For $k\in \Z^2\setminus\{0\}$, we define
\[
\Ak:=\{x\in S\setminus C_{\nabla\varphi}:\,\nabla^2\varphi(x):k^\bot\otimes k^\bot=0\}\,,
\]
write $\ck$ for the associated characteristic function, and set $\ckt=\ck+\chi_{C_{\nabla\varphi}}$.
  Then for every $k\in\Z^2\setminus\{0\}$, the function
\[
x\mapsto(1-\ckt(x))\int_Y G(x,y)\exp(2\pi\i k\cdot y)\d y
\]
is 0 almost everywhere.\\
  Rephrased in that form, it is apparent that Lemma~\ref{L:main} entails
  a characterization of two-scale limits under the \emph{nonlinear
    differential constraint} \eqref{eq:monge}.  Note that the interplay of
  two-scale convergence and \emph{linear} differential constraints is
  reasonably well understood, see e.g. \cite{Fonseca-Kroemer-10} for general results in that direction. In contrast, to our
  knowledge our result is the first treatment of a \emph{nonlinear}
  differential constraint.
\end{remark}

We are now ready for the proof of Proposition
\ref{P:1}.

\begin{proof}[Proof of Proposition~\ref{P:1}]

\step 1 Argument for (a).
\medskip

  Since $\secf^\e\wto\secf$ in $L^2(S)$, we have
  \begin{equation}
    \label{eq:36}
    \int_Y G(x,y)dy=0 \text{ for a.e. }x\in S\,. 
  \end{equation}
  Recalling the definition of $\ckt$ from Lemma \ref{L:main},  we have $\ckt(x)=0$ for all
  $k\in \Z^2\setminus\{0\}$ and
  almost every $x\in K_{\nabla u}$.
  Hence, by the conclusion of that lemma and \eqref{eq:36}, we have
  $\int_YG(x,y)\exp(2\pi \i k\cdot y)\d y=0$ for almost every~$x\in K_{\nabla u}$ and
  every $k\in\Z^2$. This implies $y\mapsto G(x,y)$ is
  identical to $0$ in $L^2(\Y)$ for almost every $x\in K_{\nabla u}$, which
  yields the claim.
\medskip

\step 2 Argument for (b).
\smallskip

Since rational directions $T\in\mathcal S^1_*$ play a special role in our argument, set $\mathcal
S_{\nabla u,*}:=\mathcal S_{\nabla u}\cap\mathcal S^1_*$. Let $B$ denote a ball with $2B\subset S$. Since $S$ can be covered by countably many of such balls, it suffices to prove identity
  \eqref{eq:25} for almost every $(x,y)\in B\times\mathcal Y$.
  Furthermore,   thanks to \eqref{L:Z:3}, it suffices to show that for
  all $T\in\mathcal S_{\nabla u, *}$ there exists $\alpha_{T}\in
  L^2(B,W^{1,2}_{T\text{-per}}(\R))$ such that
  \begin{equation}
    \label{eq:15}
    \cZ(x)G(x,y)=\sum_{T\in\mathcal S_{\nabla u, *}}\chi_{\nabla
      u,T}(x)\partial_s\alpha_T(x,T\cdot y)(T\otimes T)\qquad\text{for
      a.e. }(x,y)\in B\times\mathcal Y.
  \end{equation}
  From now on all identities hold for almost every $(x,y)\in
  B\times\mathcal Y$ or for almost every $x\in B$, respectively.

  We start our argument for \eqref{eq:15} with an application of
  Corollary~\ref{C:1}: By \eqref{C:1:2} and \eqref{C:1:3} there
  exists $\tilde\mu\in L^2(B\times\mathcal Y)$ such that
  \begin{equation*}
    \chi_{\nabla u}(x)\Big(\secf(x)+G(x,y)\Big)=\chi_{\nabla u}(x)\tilde\mu(x,y)(N^\perp(x)\otimes N^\perp(x)).
  \end{equation*}
  Due to the definition of $\chi_{\nabla u,T}$ and by \eqref{L:Z:3} we find that
  \begin{equation*}
    \cZ(x)G(x,y)=\sum_{T\in\mathcal S_{\nabla u}}\chi_{\nabla u,T}(x)\mu(x,y)(T\otimes
    T),
  \end{equation*}
  where $\mu(x,y):=\tilde\mu(x,y)-\int_Y\tilde\mu(x,y)\,dy$. Hence, in
  order to deduce \eqref{eq:15} we only need to show that
  \begin{equation}
    \label{eq:13}
    \chi_{\nabla u,T}(x)\mu(x,y)=
    \begin{cases}
          \chi_{\nabla u,T}(x)\partial_s\alpha_T(x,T\cdot y)&\text{if
          }T\in\mathcal S_{\nabla u,*},\\
          0&\text{if }T\in\mathcal S_{\nabla u}\setminus\mathcal
          S_{\nabla u,*}.
    \end{cases}
  \end{equation}
  Here comes the argument. First, we represent $\mu(x,y)$ via a Fourier-series  w.~r.~t. $y$:   \begin{equation}\label{eq:29}
    \mu(x,y)=\sum_{k\in\Z^2}a_k(x)\exp(2\pi ik\cdot y)\qquad\text{for
      some }a\in
  L^2(B,\ell^2(\Z^2)).
  \end{equation}
  From $\int_Y\mu(x,y)\,dy=0$ we deduce that $a_0=0$.  Now recall the
  definition of $\tilde\chi_k$ from Lemma~\ref{L:main} and note that
  for all
  \begin{equation}\label{eq:18}
    k\in\Z^2\setminus\{0\}\text{ with }k^\perp\cdot T\neq 0,
  \end{equation}
  we have 
  \begin{equation}\label{eq:17}
    \chi_{\nabla u,T}=(1-\tilde\chi_k)\qquad\text{a.e.~in }B.
  \end{equation}
  Hence, an application of Lemma~\ref{L:main} shows that for all $k$
  satisfying \eqref{eq:18} we have
  \begin{equation}\label{eq:16}
\chi_{\nabla u,T}(x)\int_Y\mu(x,y)\exp(2\pi ik\cdot
    y)\,dy=0,
  \end{equation}
  and thus $\chi_{\nabla u,T}a_{-k}=0$. If $T\in\mathcal
  S_{\nabla u}\setminus\mathcal S_{\nabla u, *}$, then \eqref{eq:18} is satisfied for every
  $k\in\Z^2\setminus\{0\}$ and \eqref{eq:13} follows. It remains to
  consider the case $T\in\mathcal S_{\nabla u, *}$. From \eqref{eq:29} --
  \eqref{eq:16} we learn that
  \begin{eqnarray*}
    \chi_{\nabla u,T}(x)\mu(x,y) &=&    \chi_{\nabla
      u,T}(x)\sum_{k\in\Z^2\setminus\{0\}\atop k|| T}a_k(x)\exp(2\pi
    ik\cdot y)\\
    &=& \chi_{\nabla u,T}(x)\partial_s\alpha_T(x,T\cdot y),
  \end{eqnarray*}
  where $\alpha_T\in L^2(B,W^{1,2}(S))$ is given explicitly by
  \begin{equation*}
    \alpha_T(x,s):=\chi_{\nabla
      u,T}(x)\sum_{k\in\Z^2\setminus\{0\}\atop k||
      T}\frac{a_k(x)}{2\pi i(k\cdot T)}\exp(2\pi
    i (k\cdot T)s).
  \end{equation*}
   Thanks to the elementary identity
   \begin{equation*}
     (k\cdot T)(s+k'\cdot T)=(k\cdot T)s+k\cdot k'\in (k\cdot T)s+\Z,
   \end{equation*}
   which holds for all $s\in\R$, $k\in\Z^2\setminus\{0\}$ with $k||T$,
   and $k'\in\Z^2$, we deduce that $\alpha_T(x,s)$ satisfies the
   required periodicity property in $s$, i.e.  $\alpha_T\in L^2(B,
   W^{1,2}_{T\text{-per}}(\R))$. This completes the argument for
   \eqref{eq:13}, and the proof of the proposition.

\end{proof}

Finally, we present the proofs of the auxiliary results, Lemma \ref{L:ext} and
Corollary \ref{C:1}.

\begin{proof}[Proof of Lemma~\ref{L:ext}]
  \step 1  We claim that it suffices to construct a vector field $\tilde N:B\to\mathcal
S^1$ that satisfies \eqref{L:L:1} and \eqref{L:L:2} (with $N$ replaced
by $\tilde N$) such that $F:B\to\R^{2\times 2}$, $F(x):=(\tilde
N(x)\otimes\tilde N(x))$ is continuous. Here comes the argument: Since $B$ is simply connected, there exists a continuous vector field $N:B\to\mathcal S^1$ with $F=N\otimes N$. Hence, it remains to check that $N$ satisfies \eqref{L:ext:lip}. To that end let $x,y\in B$. We need to show that
\begin{equation}\label{eq:-2}
  |N(x)-N(y)|\leq\frac{1}{\text{radius}(B)}|x-y|.
\end{equation}
We distinguish the following cases:
\begin{itemize}
\item If either $[x,N(x)]=[y,N(y)]$ or
  $[x,N(x)]\cap[y,N(y)]=\emptyset$, then $N(x)$ and $N(y)$ must be
  parallel. We argue that $N(x)=N(y)$, which means that \eqref{eq:-2}
  is trivially fulfilled. Indeed, if this were not the case, then
  $N(x)$ and $N(y)$ would be in different connected components of $\mathcal S^1\setminus\{\pm(x-y)/|x-y|\}$. By the continuity of $N$ and the fact that $[x,y]$ -- the line segment  connecting $x$ and $y$ -- is contained in $B$, there
  would have to exist $z\in[x,y]\setminus\{x,y\}$ such that $N(z)\in\{\pm (x-y)/|x-y|\}$,
  and thus $[z;N(z)]_B\cap [x,N(x)]_B=\{x\}\neq \emptyset$ in contradiction to
  eq.~\eqref{L:L:2}. 
\item If $[x,N(x)]\neq [y,N(y)]$ and $[x,N(x)]\cap[y,N(y)]\neq \emptyset$, then the lines intersect in some point $A\in\R^2$. By elementary geometry and by appealing to the continuity of $N$ as in the argument above, we deduce that
  \begin{equation*}
    \begin{array}{rll}\text{either }&N(x)=\frac{x-A}{|x-A|}\,,&N(y)=\frac{y-A}{|y-A|}\,\\
      \text{or }&N(x)=-\frac{x-A}{|x-A|}\,,&N(y)=-\frac{y-A}{|y-A|}\,\,.
    \end{array}
  \end{equation*}
  By \eqref{L:L:2} we necessarily have $A\not\in 2B$, so that (assuming without loss of generality that
  $|x-A|\leq |y-A|$)
  \begin{align*}
    |N(x)-N(y)|\leq &\left|N(x)-\frac{|y-A|}{|x-A|}N(y)\right|\\
    \leq & \frac{1}{|x-A|}|x-A-y+A|\\
    \leq & \frac{1}{\text{radius}(B)}|x-y|.
  \end{align*}
\end{itemize}

\step 2 Structure of the connected components of $C_{\nabla u}\cap B$. 
\smallskip

\noindent
Let $U$ be a connected component of $ C_{\nabla u}\cap B$. We claim that the boundary of $U$ in $B$ can be written as the
union of at most 2 disjoint line segments, and the corresponding lines do not intersect
in $2B$, that is: there exists $k\in\{0,1,2\}$, and  $x_i\in B$, $N_i\in \mathcal
S^1$ for $1\leq i\leq k$, 
 such that 
 \begin{align}
\partial U \cap B=&\bigcup_{i=1}^k [x_i;N_i]_B\,,\label{eq:23}\\
[x_i;N_i]_{2B}\cap [x_j;N_j]_{2B}=&\emptyset \text{ for }i\neq j.\label{eq:24}
\end{align}
We first define some notation that we are going to use in the argument. For  distinct
points $A,C\in\R^2$, let $\overline{AC}$ denote the line $\{A+t(C-A):t\in
\R\}$ and let $\overrightarrow{AC}$ denote the half line
$\{A+t(C-A):t\in [0,\infty)\}$. For pairwise distinct points $A,C,D\in\R^2$, let
$\angle ACD$ denote the smaller angle enclosed by the half lines
$\overrightarrow{CA}$ and $\overrightarrow{CD}$. We adopt the convention that
all such angles are positive. Let the center of $B$ be denoted by $O$. 

Now, notice that the boundary of $U$ in $B$ has to be the union of open disjoint
line segments since this is  true for the boundary of $C_{\nabla
  u}$ in $B$ by Proposition \ref{P:L}. Furthermore, the corresponding lines do
not intersect in $2B$. This proves eqs.~\eqref{eq:23} and \eqref{eq:24} for some
$k\in \N$, and it remains to show that $k\leq 2$.\\
Assume the contrary.
Then there exist three lines $L_1,L_2,L_3$ such that  (cf.~Figure \ref{fig:linecirc})
\begin{itemize}
\item $L_i\cap L_j\cap 2B=\emptyset$ for $i\neq j$
\item $L_i\cap B\neq \emptyset$ for $i=1,2,3$
\item $\bigcup_{i=1}^3L_i\cap B\subset \partial U$
\end{itemize}

\begin{figure}
\begin{center}
\includegraphics[height=5cm]{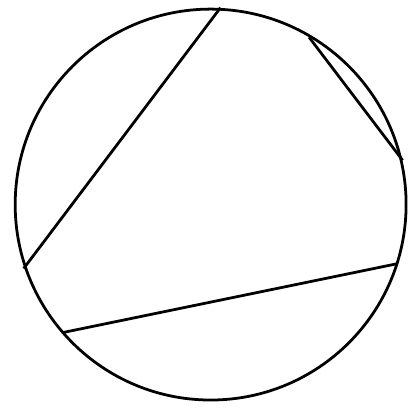}
\caption{Three line segments contained in $\partial U$.}
\label{fig:linecirc}
\end{center}
\end{figure}

Let  $m_i$, $i\in\{1,2,3\}$ be the midpoints of $L_i\cap B$.
Since $U$ is connected, either the $L_i$, $i=1,2,3$ enclose a triangle $\triangle\subset \R^2$ or two of the
lines are parallel and the third is not. \\ 
In the first case,  let $A_i$ be the  corner
of the triangle that is opposite to the side containing $m_i$, see Figure
\ref{Am}. 
\begin{figure}
\begin{center}
\includegraphics[height=5cm]{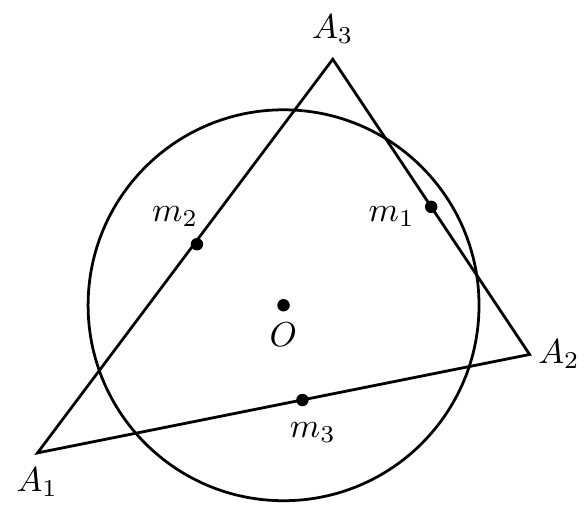}
\caption{The triangle $\triangle$ containing the line segments, and the ball
  $B$.}
\label{Am}
\end{center}
\end{figure}
Let $i,j\in\{1,2,3\}$, $i\neq j$. 
Since $m_j$ is the midpoint of
$L_j\cap B$, the line $\overline{Om_j}$ is orthogonal to $L_j$, 
and (see Figure \ref{orttri})
\[
\sin \left( \angle OA_im_j\right)=|m_j-O|/|A_i-O|<1/2\,.
\]
The latter estimate holds since $m_j\in B$ and
$A_i\not\in 2B$ by assumption. Hence the enclosed angle is less than
$\pi/6$. This is true for all pairs $i\neq j$. If $(i,j,k)$ is
some permutation of $(1,2,3)$, then
\[
\angle m_iA_jm_k\leq \angle OA_jm_i +\angle OA_jm_k\,.
\]
(Inequality  occurs if $O$ is outside $\triangle$.)
The contradiction is  obtained by using the fact that the sum of the angles
in $\triangle$ is equal to $\pi$,
\[
\pi=\angle m_1A_2m_3+\angle m_2A_1m_3+\angle m_3A_2m_1<\pi\,.
\]

\begin{figure}
\begin{center}
 \includegraphics{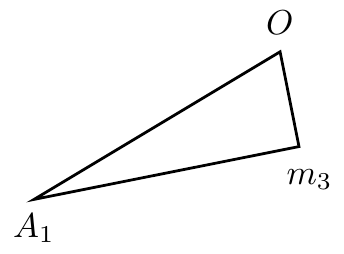}
 
 \caption{The Sine of the angle enclosed by $\protect\overrightarrow{A_1O}$
 and
   $\protect\overrightarrow{A_1m_3}$ is given by $|m_3-O|/|A_1-O|$. This ratio is smaller than
   $1/2$ since $m_3\in B$ and $A_1\not\in 2B$. Thus the angle is  smaller than $\pi/6$.}
 \label{orttri}
\end{center}
\end{figure}

In the case that two lines, say $L_1$ and $L_2$, are parallel,
let $m_i,\,i\in\{1,2,3\}$ be as
before, $A_1$ the point where $L_2$ and $L_3$ intersect, and $A_2$ the point
where $L_1$ and $L_3$ intersect,  see Figure
\ref{para}.

\begin{figure}
\begin{center}
\includegraphics{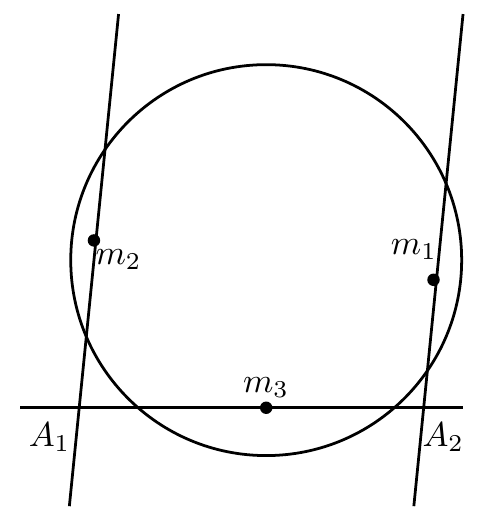}
\caption{The case of parallel line segments contained in $\partial U$.}
\label{para}
\end{center}
\end{figure}

 With the same reasoning as
before, the angles $\angle m_1A_2m_3$, $\angle m_3A_1m_2$ are both smaller than
$\pi/3$. Since $L_1$ and $L_2$ are parallel, the sum of these angles
has to be $\pi$, which produces the contradiction, and finishes the
proof of \eqref{eq:23} and \eqref{eq:24} with $k\leq 2$.\\

\step 3 Conclusion: Construction of $\tilde N$.
\smallskip

\noindent
By Step~1, to complete the proof we only need to construct a vector
field $\tilde N: B\to\mathcal S^1$ that satisfies \eqref{L:L:1} and
\eqref{L:L:2} such that $F=\tilde N\otimes\tilde N$ is continuous on
$B$. In the trivial case $C_{\nabla u}=B$ we simply set $\tilde
N=e_1$. Suppose now that $C_{\nabla u}\neq B$. We define $\tilde N$ on
$B\setminus C_{\nabla u}$ via Proposition~\ref{P:L}. The thus defined
$F=\tilde N\otimes\tilde N$ is continuous on $B\setminus C_{\nabla u}$
and $\tilde N$ satisfies \eqref{L:L:1} and \eqref{L:L:2} for $x,y\in
B\setminus  C_{\nabla u}$. On the remainder $B\cap C_{\nabla u}$ we define $\tilde N$ on each connected component $U$ separately as described next. Note that on $U$ \eqref{L:L:1} is trivially fulfilled.  Since $U\neq B$, by Step~2 the boundary $\partial U\cap B$ consists of one or two connected components. If $\partial U\cap B=[x_1;N_1]_B$ for some $x_1\in B$ and
$N_1\in\mathcal S^1$,  we set $\tilde N=N_1$ on $U$. If $\partial U\cap B=[x_1;N_1]_B\cup[x_2;N_2]_B$ for some $x_1,x_2\in B$ and
$N_1,N_2\in\mathcal S^1$, we distinguish two cases:
\begin{itemize}
\item if $N_1$ and $N_2$ are not parallel, then there exists a unique 
$A\in [x_1;N_1]\cap[x_2;N_2]$ and we set $\tilde N(y):=(A-y)/|A-y|$ for $y\in U$;
\item if $N_1$ and $N_2$ are parallel, then we set $\tilde N=N_1$.
\end{itemize}
The  thus defined vector field $\tilde N:B\to\mathcal S^1$ satisfies
\eqref{L:L:1} and \eqref{L:L:2} by construction. 
By step 1, it remains to show that $F=\tilde N\otimes \tilde N$ is continuous. We parallel 
the argument from step 1: If either $[x,\tilde N(x)]=[y,\tilde N(y)]$ or
$[x,\tilde N(x)]\cap[y,\tilde N(y)]=\emptyset$ then $F(x)=F(y)$. Otherwise, the
lines $[x,\tilde N(x)]$, $[y,\tilde N(y)]$ intersect in exactly one point
$A\in\R^2$, and 
\[
\begin{split}
  F(x)= &\frac{x-A}{|x-A|}\otimes \frac{x-A}{|x-A|}\\
  F(y)= &\frac{y-A}{|y-A|}\otimes \frac{y-A}{|y-A|}\,.
\end{split}
\]
By $A\not\in 2B$, $|F(x)-F(y)|\leq \frac{2}{\text{radius}\, B}|x-y|$, which
proves the continuity of $F$.
\end{proof}

\begin{proof}[Proof of Corollary~\ref{C:1}]
  \step 1 Argument for \eqref{C:1:1}.\smallskip

  \noindent  
  Since $N^\e$ is a vector field of unit vectors,
  and since $\operatorname{Lip}(N^\e)$ is bounded uniformly in $\e>0$,
  the sequence $N^\e$ is bounded in $W^{1,\infty}(B,\R^2)$.
  Hence, $N^\e\stackrel{*}{\wto} \tilde N$   weakly-star in
  $W^{1,\infty}$, up to a subsequence (that we do not relabel), and
  $\tilde N\in W^{1,\infty}(B,\R^2)$. Since
  $W^{1,\infty}(B,\R^2)$ is compactly embedded into the H\"older
  spaces $C^{0,\alpha}(B,\R^2)$, $0\leq\alpha<1$, the convergence
  holds uniformly and we deduce that $\tilde N(x)\in\mathcal S^1$
  almost everywhere. 
  \medskip

  \step 2 Argument for \eqref{C:1:2}.\smallskip

  \noindent
  Set $T^\e(x):=-(N^\e(x))^{\perp}$. By \eqref{L:1:2} we have
  \begin{equation}\label{eq:2}
    \secf^\e(x)=\mu^\e(x)\,T^\e(x)\otimes
    T^\e(x)\,\qquad\text{for some } \mu^\e\in L^2(B).
  \end{equation}
  The sequence $\mu^\e$ is bounded in $L^2(B)$. Hence,
  we can pass (to a further) subsequence with $\mu^\e\wtto\tilde\mu(x,y)$ two-scale in
  $L^2(B\times\Y)$. Combined with the uniform convergence
  $N^\e\to\tilde N$, \eqref{C:1:2} follows via Lemma~\ref{L:stto}.
  \medskip

  \step 3 Argument for \eqref{C:1:3}.
  \smallskip

  \noindent
  For convenience set
  $T:=-N^{\perp}$. Note that \eqref{eq:2} remains valid when the superscript  $\e$
  is dropped. By assumption we have $u^\e\wto u$ in $W^{2,2}$, and thus $\secf^\e\wto \secf$ weakly
  in $L^2(S,\R^{2\times 2})$. Since $N^\e\otimes N^\e\to \tilde
  N\otimes\tilde N$ uniformly in $B$ we obtain
  \begin{equation*}
    \int_B\big(\secf:(\tilde N\otimes\tilde
    N)\big)\varphi\,dx=\lim\limits_{\e\downarrow 0}\int_B\big(\secf^\e:( N^\e\otimes
     N^\e)\big)\varphi\,dx,
  \end{equation*}
  for all $\varphi\in L^2(B)$.   By orthogonality, the right-hand side
  vanishes, and thus
  \begin{equation}\label{eq:8}
    0=\int_B\big(\secf:(\tilde N\otimes\tilde
    N)\big)\varphi\,dx.
  \end{equation}
  The combination of identity \eqref{eq:2} (with the superscript $\e$ dropped) and \eqref{eq:8} (with
  $\varphi=\mu$) yields
  \begin{equation*}
    0=\int_{B}|\mu|^2(T\otimes T):(\tilde N\otimes\tilde
    N)\,dx=\int_{B}|\mu|^2|T\cdot\tilde N|^2\,dx.
  \end{equation*}
  Since $|\mu|^2>0$ almost everywhere in $B\setminus C_{\nabla
    u}$, the previous identity implies that $\tilde N$ and $T$ are
  orthogonal in that region, and
  thus, by the continuity of $\tilde N$ and $N=T^\perp$, we obtain
  \eqref{C:1:3}.
\end{proof}

\begin{proof}[Proof of Lemma \ref{L:loc}]
(i) We will write $N(t):=N(\Gamma(t))$.
The existence and regularity of the curve $\Gamma$ follows from a
  standard fix point argument. Since $\Gamma'(t)=-N^\perp(\Gamma(t))$ is a unit vector, we
  deduce that $\Gamma''(t)$ is orthogonal to $\Gamma'(t)$ and thus
  parallel to $N(t)$. Hence, there  exists an $L^2$ function $\kappa(t)$ such that $\Gamma''(t)=\kappa(t)N(t)$.
  We have for almost every $t$
  \begin{equation*}
    |\kappa(t)|=|\Gamma''(t)|=|\nabla
    N(\Gamma(t))\Gamma'(t)|\leq|\nabla N(\Gamma(t))|\leq
    \operatorname{Lip}(N).
  \end{equation*}
  The estimate $\operatorname{Lip}(N)\leq \frac{1}{R}$ (cf.
  \eqref{L:ext:lip}) completes the argument.\\
(ii) 
Let $(t,s),(t',s')\in Q$. Then
\begin{align*}
|\Phi(t,s)-\Phi(t',s')|\leq & |\Phi(t,s)-\Phi(t',s)|+|\Phi(t',s)-\Phi(t',s')|\\
\leq & |\Gamma(t)-\Gamma(t')|+|N(t)-N(t')||s|+|N(t')||s-s'|\\
\leq & |t-t'|+R^{-1}|t-t'|R/2+|s-s'|\\
\leq & 2|(t,s)-(t',s')|\,.
\end{align*}
This proves the first estimate in \eqref{L:10:st2:1}.
Hence, \eqref{L:L:2} implies that $\Phi$ is one-to-one. A direct calculation yields
  \begin{equation*}
    \nabla\Phi=\Big(\Gamma'(t),\;N(t)\Big)\Big(\id-s\kappa(t)e_1\otimes e_1\Big),
  \end{equation*} 
   Since $(\Gamma', N)$ is a rotation, and $|s\kappa(t)|\leq
  \frac{1}{2}$ by eq.~\eqref{L:10:st1}, we get
\begin{equation}
\frac{1}{2}\leq\det\nabla\Phi=1-s\kappa(t)\leq 2\,.\label{detphiestim}
\end{equation} 
This completes the proof of eq.~\eqref{L:10:st2:1}.
\smallskip

\noindent

A proof of the inclusion \eqref{L:10:st2:2} can be found in \cite[Remark~5]{Hornung-11a}.
For \eqref{L:5:1},  see e.~g. \cite[Proposition~1]{Hornung-11b}. Identity \eqref{L:5:2}
  follows from \eqref{L:5:1} combined with
  $|\det\nabla\Phi|=1-s\kappa(t)$.
\end{proof}

\section{Proof of Theorem~\ref{T:1}}
\subsection{Proof of Theorem~\ref{T:1} (a) \& (b) -- compactness and
  lower bound}

\begin{proof}[Proof of statement (a) -- compactness] 
In view of the coercivity assumption (Q3) and Poincar\'e's inequality, any sequence $u^\e$ with finite energy and mean zero is bounded in $W^{2,2}(S,\R^3)$. Hence, the statement follows from the observation that $W^{2,2}_{\iso}(S)$ is closed under weak convergence in
$W^{2,2}(S,\R^3)$.
\end{proof}

\begin{proof}[Proof of statement (b) -- lower bound] 
By the compactness statement (a), we may assume
without loss of generality that $u^\e,u\in W^{2,2}_{\iso}(S)$ and
\begin{alignat}{2}
  \label{P:T1:1}
  u^\e&\wto u&\qquad&\text{weakly in }W^{2,2}(S,\R^3),\\
  \label{P:T1:2}
  \secf^\e&\wto \secf&\qquad&\text{weakly in }L^{2}(S,\R^{2\times 2}),\\
  \label{P:T1:3}
  \secf^\e&\wtto \secf+G&&\text{weakly two-scale in }L^2(S\times\mathcal
  Y,\R^{2\times 2}),
\end{alignat}
where $\secf^\e$ and $\secf$ denote the second fundamental forms
associated with $u^\e$ and $u$, and $G(x,y)$ is a function in
$L^2(S\times\mathcal Y,\R^{2\times 2})$. By Lemma~\ref{L:twoscale-lsc} (a)
we have
\begin{eqnarray*}
  \lim\inf\limits_{\e\downarrow 0}\int_S
  Q(\frac{x}{\e},\secf^\e(x))\,dx &\geq& \int_{S\times
    Y}Q(y,\secf(x)+G(x,y))\,dydx.
\end{eqnarray*}
Hence, it suffices to show that
\begin{eqnarray}
  \label{eq:21}
  \int_{S\times Y}(1-\cZ(x))Q(y,\secf(x)+G(x,y))\,dydx&\geq&
  \int_{S}(1-\cZ(x))Q_{\av}(\secf(x))\,dx,\\
  \label{eq:22}
  \int_{S\times Y}\cZ(x)Q(y,\secf(x)+G(x,y))\,dydx&\geq&
  \int_{S}\cZ(x)Q_{\hom}(\secf(x))\,dx.
\end{eqnarray}
We start with \eqref{eq:21}. By \eqref{L:Z:1} we have $\{\cZ=0\}\subset C_{\nabla
  u}\cup K_{\nabla u}\cup E$ for some null set $E$. An application of Proposition~\ref{P:1} shows
that $G=0$ almost everywhere on $K_{\nabla
  u}\times\mathcal Y$, so that
\begin{eqnarray*}
  \text{[LHS of \eqref{eq:21}]}
  &\geq&
  \int_{S\setminus C_{\nabla
      u}}(1-\cZ(x))\left(\int_{Y}Q(y,\secf(x))\,dy\right)\,dx\\
  &=&
  \int_{S}(1-\cZ(x))Q_{\av}(\secf(x))\,dx.
\end{eqnarray*}
For the last identity we used that $Q_{\av}(\secf(x))=0$ almost everywhere in $C_{\nabla u}$.

It remains to prove \eqref{eq:22}. Let $\mathcal S_{\nabla
  u}\subset\mathcal S^1$ denote the set from Lemma~\ref{L:Z}, and recall that $\mathcal S_{\nabla u}$ is at
most countable. From \eqref{L:Z:3}, Fubini's theorem, and the
fact that the functions $\chi_{\nabla u,T}$ are $\{0,1\}$-valued, we deduce that
\begin{eqnarray}\notag
  &&\int_{S\times Y}\cZ(x)Q\Big(y,\secf(x)+G(x,y)\Big)\,dydx\\
  \label{eq:37}
  &=&
  \sum_{T\in\mathcal S_{\nabla u}}\int_{S}\chi_{\nabla u,T}(x)\left(\int_YQ\Big(y,\secf(x)+G(x,y)\Big)\,dy\right)\,dx.
\end{eqnarray}
From \eqref{L:1:2} and Proposition~\ref{P:1} (b), we deduce that there
exists $\mu\in L^2(S)$ and for all $T\in\mathcal S_{\nabla
  u}\cap\mathcal S^1_*$ a function $\alpha_T\in
L^2(S,W^{1,2}_{T\text{-per}}(\R))$ such that for almost every $(x,y)\in
S\times\mathcal Y$:
\begin{equation*}
  \chi_{\nabla u,T}(x)(\secf(x)+G(x,y))=
  \chi_{\nabla u,T}(x)\left\{
    \begin{aligned}
      &\mu(x)(T\otimes T)&&\text{if }T\in\mathcal S_{\nabla
        u}\setminus\mathcal S^1_*,\\
      &(\mu(x)+\partial_s\alpha_{\H{T}}(x,T\cdot y))(T\otimes T)&&\text{if }T\in\mathcal S_{\nabla
        u}\cap\mathcal S^1_*.
    \end{aligned}
  \right.  
\end{equation*}
Hence, in view of the definition of $Q_{\hom}(T\otimes T)$, see \eqref{def:hom-formula},
we have for all $T\in\mathcal S_{\nabla u}\cap\mathcal S^1_*$ and
almost every $x\in S$:
\begin{eqnarray*}
  \chi_{\nabla u,T}(x)\int_YQ\Big(y,(\secf(x)+G(x,y))\Big)\,dy&\geq&
  \chi_{\nabla u,T}(x)\mu^2(x)Q_{\hom}(T\otimes T)\\
  &=&\chi_{\nabla u,T}(x)Q_{\hom}(\mu(x)T\otimes T)\\
  &=&\chi_{\nabla u,T}(x)Q_{\hom}(\secf(x)),
\end{eqnarray*}
and similarly, for all $T\in\mathcal S_{\nabla u}\setminus\mathcal S^1_*$ and
almost every $x\in S$:
\begin{eqnarray*}
  \chi_{\nabla
    u,T}(x)\int_YQ\Big(y,(\secf(x)+G(x,y))\Big)\,dy&\geq&\chi_{\nabla
    u,T}(x)Q_{\av}(\secf(x))\\
  &\stackrel{T\not\in\mathcal S^1_*}{=}&\chi_{\nabla
    u,T}(x)Q_{\hom}(\secf(x)).
\end{eqnarray*}
Combined with \eqref{eq:37}, the claimed inequality \eqref{eq:22} follows.

\end{proof}

\subsection{Proof of Theorem~\ref{T:1} (c) - construction of recovery
  sequences}

The construction of the recovery sequence consists of two parts. In
the first part, which is the heart of the matter, we locally modify $u$ in order to recover the
oscillatory effects of homogenization. This is done on what we call ``patches'',
i.e. ``regular'' subdomains on which $u$ can conveniently be described by line of curvature
coordinates, see Definition~\ref{D:patch}. In a second part we apply an
approximation scheme due to \cite{Pakzad-04}, \cite{Hornung-11a} and
\cite{Hornung-11b}. In these works the approximation of
Sobolev isometries by smooth isometries is discussed, and as a central
step it is shown that any Sobolev isometry can be approximated by
isometries whose gradients are \textit{finitely developable}, see below for the
precise definition.

For the  definition of a ``patch'', we introduce (as in \cite{Pakzad-04}) for $u\in
W^{2,2}_{\iso}(S)$ the set $\hat C_{\nabla u}\subset C_{\nabla u}$ as the union of all connected components
$U\subset C_{\nabla u}$ with the property that $\partial U\cap S$
consists of more than two connected components. In \cite{Pakzad-04} it
is shown that the field of asymptotic directions $N$ can be
extended to $S\setminus\hat C_{\nabla u}$.

This will not quite be enough for our purposes, since we wish to
consider affine boundary conditions posed on a line segment.
In order to treat the
boundary condition \eqref{ass:BC}, we need the following variant of
this statement:

\begin{lemma}
  \label{L:RS:1}
  Let $u\in W^{2,2}_{\iso}(S)$. Then there exists a  locally Lipschitz
  continuous vector field  $N:S\setminus\hat  C_{\nabla u}\to\mathcal S^1$ such that \eqref{P:L:1} and
  \eqref{P:L:2} hold for all $x,y\in S\setminus\hat C_{\nabla u}$.
  Moreover, if $u$ satisfies \eqref{ass:BC}, then we can chose $N$
  such that 
  \begin{equation}
    \label{L:LB:1:1}
    \left\{\begin{aligned}
      &\text{either $L_{BC}\subset\hat C_{\nabla u}$},\\
      &\text{or  $L_{BC}=[x,N(x)]_S$ for some $x\in S\setminus\hat C_{\nabla u}$}.
    \end{aligned}\right.
  \end{equation}
\end{lemma}

The proof of this and the following lemmas is postponed to the end of
this section.

\begin{remark}
As stated in Remark 2 under Proposition 1 of \cite{Hornung-11a}, the choice of
the vector field $N:S\setminus \hat C_{\nabla u}\to {\mathcal S}^1$ is non-unique. The lemma above
makes a particular choice. The results of \cite{Hornung-11a} do not depend on
the choice of this vector field, cf.~again the remark just mentioned. In
particular, in the statement of Theorem \ref{thm:horn2} below, we may assume
that $N$ is the vector field constructed in Lemma \ref{L:RS:1}.
\end{remark}

\begin{definition}
  \label{D:patch}
  We call an open set $V\subset
  S\setminus\hat C_{\nabla u}$ a \textbf{patch} for $(u,N)$, if it can be parametrized by a single line of curvature chart
  $\Phi:M\to V$ in the following sense:
  \begin{enumerate}[(a)]
  \item there exist $\Gamma\in
    W^{2,\infty}([0,\ell],S\setminus\hat C_{\nabla u})$ with $\ell>0$ such that
    \begin{equation*}
      \Gamma'(t)=-N^\perp(\Gamma(t)),\qquad \Gamma'(t)\cdot\Gamma'(t')>0
    \end{equation*}
    for all $t,t'\in[0,\ell]$.
  \item $V=\Phi(M)$ where
    \begin{align*}
      &M:=\{\,(t,s)\in(0,\ell)\times\R\,:\,\Gamma(t)+sN(\Gamma(t))\in
      S\,\},\\
      &\Phi:M\to V,\qquad \Phi(t,s):=\Gamma(t)+sN(\Gamma(t)).
    \end{align*}
  \end{enumerate}
\end{definition}

The approximation of  $u\in W^{2,2}_\iso(S)$ mentioned above is carried out
with the help of two theorems below, which we quote  from 
\cite{Hornung-11a,Hornung-11b}. They deliver the desired
approximation in two steps: First, we approximate $u\in W^{2,2}_\iso(S)$ by
$u^\delta\in W^{2,2}_\iso(S)$ such that $\nabla u^\delta$ is
\textbf{finitely developable}. This means that $\hat
C_{\nabla u^\delta}$ consists of finitely many connected components, and each
connected component $U\subset \hat C_{\nabla u^\delta}$ has the property that
$\partial U\cap S$ consists of finitely many connected components.\\
In the second approximation step, $u\in W^{2,2}_\iso(S)$ with finitely
developable gradient is approximated by a map $u^{\delta'}\in W^{2,2}_\iso(S)$, with the property
that it can be parametrized by finitely many patches.

\begin{proposition}[\cite{Hornung-11b}, Proposition 5]
\label{thm:horn1}
Let $u\in W^{2,2}_\iso(S)$. Then for every $\delta>0$ there exists $u^\delta\in
W^{2,2}_\iso(S)$ with the following properties:
\begin{itemize}
\item[(i)] The gradient $\nabla u^\delta$ is finitely developable.
\item[(ii)] $u^\delta=u$ on the set
  \begin{eqnarray*}
    S_\delta&:=&\bigcup\left\{\,[x;N(x)]_S\,:\,x\in E_\delta\setminus
        \hat C_{\nabla u}\right\}\cup\\
  &&\bigcup\left\{\,U\,:\,\text{$U$ is
      a connected component of $\hat C_{\nabla u}$ with $U\cap E_\delta\neq\emptyset$}\right\}\,,
  \end{eqnarray*}
where $E_\delta:=\{\,x\in S\,:\,\dist(x,\partial S)>\delta\,\}$.
Moreover, $u^\delta$ is affine on every connected component of $S\setminus \bar
S_\delta$.
\item[(iii)] $u^\delta\to u$ strongly in $W^{2,2}(S;\R^3)$ as $\delta\dto 0$.
\end{itemize}
\end{proposition}
\begin{theorem}[\cite{Hornung-11a}, Theorem 2]
\label{thm:horn2}
Let $u\in W^{2,2}_\iso(S)$ with finitely developable gradient, let $V_1,\dots,V_m$ be the connected
components of $\hat C_{\nabla u}$ and let $N:S\setminus \hat C_{\nabla u}\to{\mathcal S}^1$
be the vector field associated to $u$ via Lemma \ref{L:RS:1}. Then for all $\delta>0$ there exists $n\in\N$
with $n\geq m$ and curves $\Gamma^{(k)}\in W^{2,\infty}([0,T_k];S\setminus\hat
C_{\nabla u})$ for $k=m+1,\dots,n$, such that, with
\[
V_k=S\cap \{\Gamma^{(k)}(t)+sN(\Gamma^{(k)}(t)):t\in(0,T_k), s\in\R\},\quad
k=m+1,\dots,n\,,
\]
the following holds true:
\begin{itemize}
\item[(i)]$N(\Gamma^{(k)}(t))\cdot(\Gamma^{(k)})'(t)=0$ for $k=m+1,\dots,n$,
  $t\in [0,T_k]$.
\item[(ii)] We have
\[
E_\delta:=\{\,x\in S\,:\,\dist(x,\partial S)>\delta\,\}\subset
\text{int}\left(\cup_{k=1}^n \bar V_k\right)\,.
\]
\item[(iii)] Whenever $j,k\in\{1,\dots,n\}$ with $j\neq k$, then
\[
V_j\cap V_k=\emptyset \,.
\]
\end{itemize}
\end{theorem}


After having collected these results  from the literature, we come to the heart of the recovery sequence construction -- the construction on a single patch.

\begin{lemma}[Construction on a single patch]
  \label{L:RS:2}
  Let $V$ be a patch for $(u,N)$. Then there exists a sequence $u^\e\in
  W^{2,2}_{\textrm{iso}}(S)$ such that
  \begin{subequations}
    \begin{eqnarray}
      \label{L:RS:2:a}
      &&\lim\limits_{\e\downarrow
        0}\int_VQ(\tfrac{x}{\e},\secf^\e(x))\,
      dx\\\notag
      &&\qquad=\int_{V}(1-\chi_{\nabla
        u}(x))Q_{\av}(\secf(x))+\chi_{\nabla u}(x) Q_{\ho}(\secf(x))\,dx,\\
      \label{L:RS:2:b}
      &&\secf^\e\wto \secf\mbox{ weakly in }L^{2}(V),\\
      \label{L:RS:2:c}
      &&\text{$u^\e$ is affine on each line segment $[x;N(x)]_S$,
        $x\in \overline V\cap S$.}
    \end{eqnarray}
  \end{subequations}
\end{lemma}

As already announced, the preceding lemma will be combined with Theorem \ref{thm:horn2} for the
construction of the recovery sequence for the case of $u\in W^{2,2}_\iso(S)$
with finitely developable gradient:

\begin{lemma}[Construction in the finitely developable case]
  \label{L:RS:3}
  Let $u\in W^{2,2}_{\iso}(S)$  such that $\nabla u$ is finitely developable. Then there exists a
  sequence $u^\e\in W^{2,2}_{\iso}(S)$ such that
  \begin{subequations}
    \begin{eqnarray}
      \label{L:RS:3:1}
      &&\lim\limits_{\e\downarrow
        0}\|u^\e-u\|_{L^2(S)}=0,\\
      \label{L:RS:3:2}
      &&\lim\limits_{\e\downarrow 0}\mathcal E^\e(u^\e)=\mathcal E^0(u),\\
      \label{L:RS:3:3}
      &&\text{if $u$ satisfies \eqref{ass:BC}, then $u^\e$ satisfies \eqref{ass:BC}}.
    \end{eqnarray}
  \end{subequations}
\end{lemma}


The construction of the recovery sequence for arbitrary $u\in W^{2,2}_\iso(S)$
satisfying the boundary condition \eqref{ass:BC} is then achieved by combining
Lemma \ref{L:RS:3} with Proposition \ref{thm:horn1}. This is what we will do
next; the proof of the theorem is followed by the proofs of the auxiliary
results above.

\begin{proof}[Proof of Theorem~\ref{T:1} (c)]
  We only need to consider the case with prescribed boundary conditions, since
  otherwise we might artificially introduce boundary conditions by introducing a line segment $L_{BC}$ on which $u$ is affine. Let
  $N:S\setminus \hat C_{\nabla u}\to\mathcal S^1$ be as in
  Lemma~\ref{L:RS:1}. We use Proposition \ref{thm:horn1} to  approximate $u$ by
  $u^\delta\in W^{2,2}_\iso(S)$ with finitely developable gradient. 
  We also adapt the definitions of
$E_\delta, S_\delta$ from the statement of that proposition. 
  For the treatment of the boundary conditions, we shall always assume that $\delta>0$ is so small that
  $L_{BC}\cap E_\delta\neq\emptyset$. 
  Note that
  \begin{equation}
    \label{eq:51}
    E_\delta\subset S_\delta\qquad\text{ and
          }\qquad     L_{BC}\subset S_{\delta}.
  \end{equation}
  The first inclusion directly follows from the definition of
  $S_\delta$. The argument for the second inclusion is postponed to
  the end of this proof. \\
  By Proposition \ref{thm:horn1}, we have $\lim_{\delta\downarrow
    0}\|u^\delta-u\|_{L^2(S)}=0$, 
  \begin{equation}\label{eq:52}
    u^\delta=u\text{ in }S_\delta,
  \end{equation}
  and $u^\delta$ is affine
  on each connected component of $S\setminus \bar S_\delta$. Note
  that the latter implies that
  \begin{equation}\label{eq:55}
    |\secf^\delta(x)|\leq|\secf(x)|\qquad\text{a.e.~in 
    }S.
  \end{equation}
  Since $u$ satisfies \eqref{ass:BC}, it follows from the second
  inclusion in \eqref{eq:51} and \eqref{eq:52} that $u^\delta$
  satisfies \eqref{ass:BC}. Furthermore,
  \eqref{eq:52} implies that $\chi_{\nabla
    u^\delta}=\chi_{\nabla u}$ and $\secf^\delta=\secf$ almost
  everywhere on $S_\delta$. Hence,
  \begin{eqnarray*}
    &&\mathcal E^0(u^\delta)-\mathcal E^0(u)\\
    &=&\int_{S\setminus S_\delta}(1-\chi_{\nabla
      u^\delta}(x))Q_{\av}(\secf^\delta(x))+\chi_{\nabla
      u^\delta}(x)Q_{\ho}(\secf^\delta(x))\,dx\\
    &&-\int_{S\setminus S_\delta}(1-\chi_{\nabla
      u}(x))Q_{\av}(\secf(x))+\chi_{\nabla
      u}(x)Q_{\ho}(\secf(x))\,dx.    
  \end{eqnarray*}
  Because of $S\setminus S_{\delta}\subset S\setminus E_\delta$ (cf.
  \eqref{eq:51}), $0\leq Q_{\ho}(F)\leq Q_{\av}(F)\leq
  \frac{1}{\alpha}|F|^2$ (cf. \eqref{ass:Q}), and  \eqref{eq:55},
  we estimate
  \begin{equation*}
    |\mathcal E^0(u^\delta)-\mathcal E^0(u)|\leq
    \frac{2}{\alpha}\int_{S\setminus E_{\delta}}|\secf(x)|^2\,dx,
  \end{equation*}
  and thus conclude
  \begin{equation}\label{eq:53}
    \lim\limits_{\delta\downarrow 0}\left(\|u^{\delta}-u\|_{L^2(S)}+\big|\mathcal
    E^0(u^{\delta})-\mathcal E^0(u)\big|\right)=0.
  \end{equation}
  Next, we apply Lemma~\ref{L:RS:3}: For  each $\delta>0$ there exists
  a sequence $u^{\delta,\e}\in W^{2,2}_{\iso}(S)$ such that each
  $u^{\delta,\e}$ satisfies \eqref{ass:BC}, and 
  \begin{equation}\label{eq:50}
    \lim\limits_{\e\downarrow 0}\left(\|u^{\delta,\e}-u^\delta\|_{L^2(S)}+\big|\mathcal
    E^\e(u^{\delta,\e})-\mathcal E^0(u^\delta)\big|\right)=0.
  \end{equation}
  Combined with \eqref{eq:53} we get
  \begin{equation*}
    \lim\limits_{\delta\downarrow 0}\lim\limits_{\e\downarrow 0}\left(\|u^{\delta,\e}-u\|_{L^2(S)}+\big|\mathcal
    E^\e(u^{\delta,\e})-\mathcal E^0(u)\big|\right)=0.
  \end{equation*}
  By a standard diagonalization argument due to Attouch (see
  \cite[Corollary~1.16]{MR773850}), there exists a map $\e\mapsto
  \delta(\e)\in\N$ such that
  \begin{equation*}
    \lim\limits_{\delta\downarrow 0}\lim\limits_{\e\downarrow 0}\left(\|u^{\e}-u\|_{L^2(S)}+\big|\mathcal
    E^\e(u^{\e})-\mathcal E^0(u)\big|\right)=0.
  \end{equation*}
  Moreover, since each $u^{\delta,\e}$ satisfies \eqref{ass:BC}, the
  diagonal sequence $u^\e$ satisfies \eqref{ass:BC} as well.

  To complete the proof, it remains to prove the second inclusion
  in \eqref{eq:51}, i.e. $L_{BC}\subset S_\delta$. If
  $L_{BC}\subset\hat C_{\nabla u}$, then there exists a connected
  component $U\subset\hat C_{\nabla u}$ that contains $L_{BC}$, and
  from  $\emptyset\neq L_{BC}\cap E_\delta\subset U\cap E_{\delta}$ we
  deduce that $S_{\delta}\supset U\supset L_{BC}$ as claimed.
  Likewise, if $L_{BC}\not\subset\hat C_{\nabla u}$, then there exists
  $x\in L_{BC}\cap (S\setminus\hat C_{\nabla u})$. From
  Lemma~\ref{L:RS:1} we infer that  $L_{BC}=[x,N(x)]_S$. Since
  $L_{BC}\cap E_\delta\neq\emptyset$ we deduce $L_{BC}\subset
  S_\delta$ from the definition of $S_\delta$.
  
\end{proof}

In the remainder of this section we present the proofs of
Lemma~\ref{L:RS:1} -- Lemma~\ref{L:RS:3}.

\begin{proof}[Proof of Lemma~\ref{L:RS:1}]
This is very similar to step 3 in the proof of Lemma \ref{L:ext}, and we are
going to be brief. We  need to construct $N$ on every connected component of
$C_{\nabla u}\setminus \hat C_{\nabla u}$. Let $U$ be such a connected
component.\\ 
By \eqref{ass:BC} we have either $L_{BC}\cap U=\emptyset$ or $L_{BC}\subset U$.
First suppose that $L_{BC}\cap U=\emptyset$. Let $L_1,L_2$ be the two
connected components of $\partial U\cap S$. Since $L_1,L_2\subset S\setminus
C_{\nabla u}$, $N$ is defined there, and takes values $N_1,N_2$ respectively.
\begin{itemize}
\item if $N_1$ and $N_2$ are not parallel, then there exists a unique 
$A\in [x_1;N_1]\cap[x_2;N_2]$ and we set $ N(y):=(A-y)/|A-y|$ for $y\in U$;
\item if $N_1$ and $N_2$ are parallel, then we set $N(y)=N_1$ for $y\in U$.
\end{itemize}
Now suppose $L_{BC}\subset U$. Choose $x,\bar N$ such that $L_{BC}=[x;\bar N]_S$, and
set $N(x)=\bar N$ on $L_{BC}$.
Then  subdivide $U$ into the two connected
components of $U\setminus L_{BC}$, and carry out the construction from the
previous case. \\
In this way, we obtain a vector field $N:S\setminus \hat C_{\nabla u}\to {\mathcal S}$ with the
property that $N\otimes N$ is locally Lipschitz (cf.~the proof of Lemma \ref{L:ext}). Since every connected component $U$ of
$S\setminus \hat C_{\nabla u}$ is simply connected, there exists a continuous
lifting $\tilde N:U\to{\mathcal S}^1$. This defines the wished for vector field.
\end{proof}

\begin{proof}[Proof of Lemma~\ref{L:RS:2}]
  Let $\Gamma\in W^{2,\infty}([0,\ell],S\setminus \hat C_{\nabla u})$
  and $\Phi:M\to V$ be associated with $V$ according
  to Definition~\ref{D:patch}.  Set $L_1:=[\Gamma(0),N(\Gamma(0)]_S$
  and note that $L_1$ is one of the two connected
  components of $\partial V\cap S$. To simplify the presentation, we say that an
  isometry $v\in W^{2,2}_{\iso}(V)$ satisfies property (A), if
  \begin{equation}\tag{A}\label{eq:A}
    \left\{\begin{aligned}
      &v\text{ is affine on each line segment }[x,N(x)]_S\text{ for all
      }x\in\overline V\cap S,\\
      &v=u\text{ and }\nabla v=\nabla u\text{ on }L_1
    \end{aligned}\right.
  \end{equation}
  By virtue of the definition of $N$, see Lemma~\ref{L:RS:1},  $u$ itself
  satisfies property \eqref{eq:A}.
  \medskip

  \step 1 A reduction step.

  We claim that it suffices to prove the following statement: 
  \begin{itemize}
  \item[(S)]
    For arbitrary $J\in\N$, mutually non-parallel vectors
    $T_1,\ldots,T_J\in\mathcal S^1$, and functions $\alpha_j\in
    C^\infty_{T_j\text{-per}}(\R)$, $j=1,\ldots,J$, there exists a
    sequence $v^\e\in W^{2,2}_{\iso}(V)$ such that $v^\e$ satisfies
    property \eqref{eq:A} and the associated fundamental form satisfies
      \begin{align}
        \label{eq:46}
        &\secf^\e\stto \Big(1+\sum_{j=1}^J\chi_{\nabla
          u,T_j}(x)\alpha_j'(T_j\cdot y)\Big)\secf(x).
      \end{align}
  \end{itemize}
  Here comes the argument. Recall the definition of ${\mathcal S}_{\nabla u}$
  from Lemma~\ref{L:Z}.  Let $T_1,T_2,\ldots$ be an enumeration of
  ${\mathcal S}_{\nabla u}$. By definition we have $\chi_{\nabla u}^*(x)=\sum_{j=1}^\infty\chi_{\nabla
    u,T_j}(x)$ for almost every $x\in S$, and thus $\lim\limits_{J\uparrow\infty}\int_{S}|(\sum_{j=1}^J\chi_{\nabla
    u,T_j})-\chi_{\nabla u}^*|Q_{\ho}(\secf)\,dx=0$. Therefore, for every
  $\delta>0$ we can find $J^\delta>0$ and functions $\alpha_{\delta,j}\in
  C^\infty_{T_j\text{-per}}(\R)$, $j=1,\ldots,J^\delta$, such that
  \begin{equation}\label{eq:42}
    \begin{aligned}
      &\int_V\Big|\chi_{\nabla u}(x)-\Big(\sum_{j=1}^{J^\delta}\chi_{\nabla u,T_j}(x)\Big)\Big|Q_{\av}(\secf(x))\,dx\leq\delta,\\
      &\Bigg|\int_V\chi_{\nabla
        u}(x)Q_{\ho}(\secf(x))-\sum_{j=1}^{J^\delta}\chi_{\nabla
        u,T_j}(x)\int_YQ(y,\alpha_{\delta,j}'(T_j\cdot
      y)\secf(x))\,dydx\Bigg|<\delta.
    \end{aligned}
  \end{equation}
  By assumption (S), there exists a sequence
  $v^{\delta,\e}\in W^{2,2}_{\iso}(S)$ with 
  \begin{equation}
    \label{eq:45}
    v^{\delta,\e}=u\text{ and }\nabla v^{\delta,\e}=\nabla u\text{
      on }L_1,
  \end{equation}
  and 
  \begin{equation}\label{eq:43}
    \secf^{\delta,\e}\stto \Big(1+\sum_{j=1}^{J^\delta}\chi_{\nabla
      u,T_j}(x)\alpha_{\delta,j}'(T_j\cdot y)\Big)\secf(x)\qquad\text{as
    }\e\downarrow 0.
  \end{equation}
  We finally claim that the sought sequence $u^\e$ can be obtained
  as a diagonal sequence of $v^{\delta,\e}$. To that end set
  \begin{eqnarray*}
    e^{\delta,\e}&:=&\int_V Q(\tfrac{x}{\e},\secf^{\delta,\e}(x))\,dx,\\
    e^0&:=&\int_{V}(1-\chi_{\nabla
      u}(x))Q_{\av}(\secf(x))+\chi_{\nabla u}(x) Q_{\ho}(\secf(x))\,dx,
  \end{eqnarray*}
  and consider
  \begin{eqnarray*}
    c(\delta,\e):=\|v^{\delta,\e}-u\|_{L^2(S)}  +|e^{\delta,\e}-e^0|.
  \end{eqnarray*}
  We shall prove that
  \begin{equation}\label{eq:41}
    \limsup\limits_{\delta\downarrow 0}\limsup\limits_{\e\downarrow 0}c(\delta,\e)=0.
  \end{equation}
  Indeed, \eqref{eq:43}
  implies that $\secf^{\delta,\e}\wto \secf$ weakly in
  $L^2(V)$ as $\e\downarrow 0$. Combined with\eqref{eq:45} we deduce
  that $v^{\delta,\e}\to u$ strongly in $L^2(V)$ as $\e\downarrow 0$. It remains to show $\lim_{\delta\downarrow0}\lim_{\e\downarrow
    0}e^{\delta,\e}=e^0$. 
  From Lemma~\ref{L:twoscale-lsc} (b) and \eqref{eq:43} we get
  \begin{eqnarray*}
    \lim\limits_{\e\downarrow 0}e^{\delta,\e}&=&\int_{V\times Y}Q\Bigg(y,\Big(1+\sum_{j=1}^{J^\delta}\chi_{\nabla
      u,T_j}(x)\alpha_{\delta,j}'(T_j\cdot y)\Big)\secf(x)\Bigg)\,dydx\\
    &=&\int_{V\times Y}\Big(1-(\sum_{j=1}^{J^\delta}\chi_{\nabla u,T_j}(x))\Big)Q\Big(y,\secf(x)\Big)\,dydx\\
    &&+\int_{V\times Y}\sum_{j=1}^{J^\delta}\chi_{\nabla
      u,T_j}(x)Q\Big(y,\Big(1+\alpha_{\delta,j}'(T_j\cdot
    y)\Big)\secf(x)\Big)\,dydx.
  \end{eqnarray*}
  Combined with \eqref{eq:42}, \eqref{eq:41} follows.

  Finally, we deduce from \eqref{eq:41}, by appealing to a
  standard diagonalization argument  (see \cite{MR773850}), that there exists a map $\e\mapsto \delta(\e)$ such that
  $c(\delta(\e),\e)\to 0$ as $\e\downarrow 0$. Hence, the diagonal sequence
  $u^\e:=v^{\delta(\e),\e}$ strongly converges in $L^2(V)$ to $u$, and its
  energy satisfies $\lim\limits_{\e\downarrow0}\int_VQ(\tfrac{x}{\e},\secf^\e(x))\,dx=e^0$.
  Since this especially implies that the associated sequence of fundamental
  forms $\secf^\e$ is bounded in $L^2(V)$, we can upgrade the
  convergence of $u^\e$ and deduce that $u^\e\wto u$ weakly in
  $W^{2,2}(V)$ as claimed. This in particular implies that
  $\secf^\e\wto\secf$ in weakly $L^2(V)$. Moreover, since each
  $u^{\delta(\e),\e}$ satisfies property \eqref{eq:A}, the same is true for  $u^\e$.
  \bigskip

  The rest of the proof is devoted to show statement (S) in Step~1. 

  \step 2 Line of curvature parametrisation of $u\vert_V$.
  \smallskip
  
  Recall that
  \begin{align*}
    &\Phi(t,s):=\Gamma(t)+sN(t),\qquad     N(t):=N(\Gamma(t)),\qquad T(t):=-N^\perp(\Gamma(t)).
  \end{align*}
  Following \cite{Hornung-11a} we introduce the  framed curve $(\gamma,R):[0,\ell]\to\R^3\times SO(3)$
  \begin{align*}
    &\gamma(t):=u(\Gamma(t)),\qquad\nu(t):=(\nabla u(\Gamma(t))N(t),\qquad
    n:=\gamma'(t)\wedge\nu(t),\\
    &R(t):=(\gamma'(t),\nu(t),n(t))^t.
  \end{align*}
  Then a direct computation shows that (e.g. see \cite[Proposition~1]{Hornung-11a})
  \begin{eqnarray*}
    u(\Phi(t,s))&=&\gamma(t)+s\nu(t),\\
    \nabla u(\Phi(t,s))&=&\gamma'(t)\otimes T(t)+\nu(t)\otimes N(t),\\
    \secf(\Phi(t,s))&=&\frac{\kappa_n(t)}{1-s\kappa(t)}(T(t)\otimes T(t)),
  \end{eqnarray*}
  with scalar  curvatures
  \begin{equation*}
    \kappa(t):=\Gamma''\cdot N,\qquad \kappa_n(t):=\gamma''(t)\cdot n(t),
  \end{equation*}
  and the frame $R$ is the unique solution in
  $W^{1,2}((0,\ell),SO(3))$ to the system
  \begin{equation*}
    R'=\left(
      \begin{array}{ccc}
        0&\kappa&\kappa_n\\
        -\kappa&0&0\\
        -\kappa_n&0&0
      \end{array}
    \right)R,\qquad R(0)=(\gamma'(0),\nu(0),n(0))^t.
  \end{equation*}

  \step 3 Manipulation of $\kappa_n$.
  \smallskip

  We claim that for any $\theta\in L^\infty([0,\ell])$ there exists
  $u_\theta\in W^{2,2}_{\iso}(V)$ satisfying property \eqref{eq:A},
  and
  \begin{equation}\label{eq:30}
    \secf_\theta(\Phi(t,s))=\frac{(1+\theta(t))\kappa_n(t)}{1-s\kappa(t)}(T(t)\otimes T(t)).
  \end{equation}
  Indeed, this follows from \cite[Proposition~2]{Hornung-11a}. For the
  convenience of the reader we briefly recall the
  construction: Let $R_\theta\in
  W^{1,2}((0,\ell),SO(3))$ be the unique solution to
  \begin{equation*}
    R'_\theta=\left(
      \begin{array}{ccc}
        0&\kappa&(1+\theta)\kappa_n\\
        -\kappa&0&0\\
        -(1+\theta)\kappa_n&0&0
      \end{array}
    \right)R_\theta,\qquad R_\theta(0)=(\gamma'(0),\nu(0),n(0))^t,
  \end{equation*}
  and define $\gamma_\theta$, $\nu_\theta$, $n_\theta$ via
  \begin{equation*}
    R_\theta=(\gamma'_\theta,\nu_\theta,n_\theta),\qquad \gamma_\theta(0)=\gamma(0).
  \end{equation*}
  Now the isometry $u_\theta:V\to\R^3$ is given by
  \begin{equation*}
    u_\theta(\Phi(t,s))=\gamma_\theta(t)+s\nu_\theta(t)
  \end{equation*}
  and its fundamental form satisfies \eqref{eq:30}. By construction
  $u_\theta$ satisfies property \eqref{eq:A}.
  \medskip

  \step 4 Proof of statement (S).
  \smallskip
  
  Let $J\in\N$, $T_1,\ldots,T_J$ and $\alpha_{j}$ as in statement (S). For $\e>0$ we define the function
  \begin{equation*}
    \theta^\e(x):=\sum_{j=1}^J\chi_{\nabla u,T_j}(x)\alpha_{j}'(\tfrac{T_j\cdot x}{\e}).
  \end{equation*}
  Note that we have $\theta^\e\in L^\infty$, since the $\alpha_j$'s
  are smooth and the sum is finite. Since $N(\Phi(t,s))$ is independent of $s$, the function
  \begin{equation*}
    \tilde\theta^\e(t,s):=\theta^\e(\Phi(t,s)),\qquad (t,s)\in M
  \end{equation*}
  is independent of $s$.
  Hence, an application of Step~3 shows that there exists an isometry $u^\e=u_{\theta^\e}$
  satisfying property \eqref{eq:A} and
  \begin{eqnarray*}
    \secf^\e(\Phi(t,s))&=&\frac{(1+\tilde\theta^\e(t))\kappa_n(t))}{1-s\kappa(t)}T(t)\otimes
    T(t).
  \end{eqnarray*}
  With $\Phi(t,s)=x$, this can be rewritten as
  \begin{equation}
    \secf^\e(x)=(1+\theta^\e(x))\secf(x).
  \end{equation}
  Since $y\mapsto\alpha_{j}'(T_j\cdot y)$ is a $\mathcal Y$-periodic function, we have
  \begin{equation*}
    \chi_{\nabla u,T_j}(x)\alpha'_{j}(\tfrac{T_j\cdot x}{\e})\secf(x)\stto
    \chi_{\nabla u,T_j}(x)\alpha'_{j}(T_j\cdot y)\secf(x)
  \end{equation*}
  strongly two-scale in $L^2(V\times\mathcal Y)$ for $j=1,\ldots,J$.
  Hence, \eqref{eq:46} follows by superposition.
  
\end{proof}

\begin{proof}[Proof of Lemma~\ref{L:RS:3}]
  We only need to consider the case with prescribed boundary conditions. Let $N:S\setminus \hat C_{\nabla u}\to\mathcal S^1$ be as
  in Lemma~\ref{L:RS:1}.  Here and below we assume that $\delta>0$ is so small that
  $L_{BC}\cap E_\delta\neq\emptyset$, where $E_\delta:=\{\,x\in
  S\,:\,\dist(x,\partial S)>\delta\,\}$.

  By appealing to a diagonalization argument similar to the one in the proof of
  Theorem~\ref{T:1} (c), we only need to prove the following statement:
  For all $\delta>0$ there exists a sequence $u^{\delta,\e}\in
  W^{2,2}_{\iso}(S)$ such that $u^{\delta,\e}$ satisfies \eqref{ass:BC} and
  \begin{equation}\label{eq:1}
   \limsup_{\delta\dto 0} \limsup\limits_{\e\downarrow
      0}\left(\|u^{\delta,\e}-u\|_{L^2(S)}+|\mathcal
      E^\e(u^{\delta,\e})-\mathcal E^0(u)|\right)=0\,.
  \end{equation}
\newcommand{\Vd}[1]{V^{(\delta)}_{#1}}
\newcommand{\Ld}{{{\mathcal L}^{(\delta)}}}
  Let us explain the construction of $u^{\delta,\e}$. By assumption,  $\nabla u$
  is finitely developable, and we may apply Theorem \ref{thm:horn2}.
Hence,
there exists a finite
  number of mutually disjoint patches
  $\Vd 1 ,\ldots,\Vd {m(\delta)}$  such that
  \begin{equation}\label{eq:38}
    E_\delta\setminus\hat C_{\nabla u}\subset
    \bigcup_{k=1}^{m(\delta)}\overline{\Vd k}=:V_\delta\,.
  \end{equation}
 In view of Definition~\ref{D:patch}, the boundary $\partial \Vd {k}\cap
  S$ of each patch $\Vd {k}$ consists of two connected components. They are line segments of the form $[x,N(x)]_S$. Define
  \begin{equation*}
    \Ld:=\big\{\,L\,:\,L\text{ is a connected component of }\partial
    {\Vd k} \cap S\text{ for some }1\leq k\leq m(\delta)\,\}\cup \{L_{BC}\},
  \end{equation*}
  and note that $u$ is affine on each $L\in \Ld$. We divide the rest of the argument into two steps.
  \smallskip

  \step 1 In this step $\delta$ is fixed. Hence we write $m(\delta)=m$,
    $\Vd{k}=V_k$, $\Ld=\mathcal L$. Also, the objects we introduce here will
    depend on $\delta$, but we are going to suppress the
    superscript $\delta$  to alleviate the notation. 
  Set $V_0:=\emptyset$. We claim that for $k=0,\ldots,m$ there
  exists a sequence $u^{\e}_k\in W^{2,2}_{\iso}(S)$ such that
  \begin{subequations}
    \begin{align}
      \label{L:RS:3:2a}
      &\text{  $u^{\e}_k$ is affine
        on each $L\in\mathcal L$},\\
      \label{L:RS:3:2b}
      &\secf^\e_k=\secf^\e_{k-1}\text{ a.e.~on
      }S\setminus
      V_k\quad (\text{for }k>0)\\
      \intertext{for all $\e>0$, and}
      \label{L:RS:3:2c}
      &\secf^\e_k\wto\secf\text{ weakly in }L^2(S)\text{ as $\e\downarrow
        0$},\\
      \label{L:RS:3:2d}
      &\lim\limits_{\e\downarrow 0}\int_{V_k}Q(\tfrac{x}{\e},\secf^\e_k(x))\,dx\\\notag
      &\qquad =\int_{V_k}(1-\chi_{\nabla
        u}(x))Q_{\av}(\secf(x))+\chi_{\nabla u}(x)Q_{\ho}(\secf(x))\,dx.
    \end{align}
  \end{subequations}
  We construct $u^{\e}_k$ inductively. The trivial
  sequence $u^{\e}_0:=u$ clearly satisfies
  \eqref{L:RS:3:2a} -- \eqref{L:RS:3:2d} for $k=0$.  Now assume that these
  properties are satisfied for some fixed index $0\leq k<m$ and a
  sequence $u^{\e}_k$. We apply Lemma~\ref{L:RS:2} to the
  patch $V_{k+1}$ and obtain a sequence $v^\e\in W^{2,2}_{\iso}(V_{k+1})$
  satisfying \eqref{L:RS:2:a} -- \eqref{L:RS:2:c}. 

  In the following we define $u^{\e}_{k+1}$ by ``merging'' $v^\e$ and
  $u^{\e}_{k}$. To that end let $\e>0$ be fixed for a moment. We claim that there exists $\tilde v\in W^{2,2}_{\iso}(S)$ that coincides with $v^\e$ on $V_{k+1}$, and is equal (up to a rigid motion) with $u^{\e}_{k}$ on
  each connected component of $S\setminus V_{k+1}$. Indeed, since $V_{k+1}$ is
  a patch, its boundary $\partial V_{k+1}\cap S$ consists of two
  line-segments $L_1,L_2\in\mathcal L$. Furthermore, due to the
  convexity of $S$, the set $S\setminus
  \overline V_{k+1}$ consists of two connected components $U_1$ and $U_2$. By \eqref{L:RS:2:c} and
  \eqref{L:RS:3:2a} the functions $v^\e$ and $u^{\e}_k$ are affine on
  $L_1$ and $L_2$. Hence, there exist rigid motions
  $\varphi_1,\varphi_2:\R^3\to\R^3$ such that 
  \begin{equation*}
    u^{\e}_{k+1}(x):=
    \begin{cases}
      \varphi_1\circ u^{\e}_k(x)&\text{if }x\in \overline U_1\cap S\\
      v^\e(x)&\text{if }x\in V_{k+1}\\
      \varphi_2\circ u^{\e}_k(x)&\text{if }x\in \overline U_2\cap S
    \end{cases}
  \end{equation*}
  defines a function in $W^{2,2}_{\iso}(S)$. We claim that  for each
  $L\in\mathcal L$
  \begin{equation}
    \label{eq:34}
    u^{\e}_{k+1}\text{ is affine on }L.
  \end{equation}
  For the argument we distinguish the two
  cases $L\cap V_{k+1}\neq\emptyset$ and $L\cap
  V_{k+1}=\emptyset$. In the latter case, the claim directly follows
  from property \eqref{L:RS:3:2a} and the fact that affine maps remain
  affine under composition with a rigid motion. Since the patches are mutually
  disjoint, and lines in $\mathcal L$ do not intersect, the case $L\cap V_{k+1}\neq\emptyset$ is only possible, if $L=L_{BC}$. Hence, there exists $x_0\in
  V_{k+1}\cap L_{BC}$. Since $V_{k+1}$ is a patch, $x_0$ necessarily belongs
  to $S\setminus\hat C_{\nabla u}$, and thus $L_{BC}=[x_0,N(x_0)]_{S}$
  due to the construction of $N$ (see Lemma~\ref{L:RS:1}). Now the
  claim follows from \eqref{L:RS:2:c}.

  It remains to check that $u^{\e}_{k+1}$ satisfies properties
  \eqref{L:RS:3:2b} -- \eqref{L:RS:3:2d}. Since the composition
  with a rigid motion does not change the second fundamental form, 
  $u^{\e}_{k+1}$ satisfies \eqref{L:RS:3:2b}, and properties
  \eqref{L:RS:3:2c} and \eqref{L:RS:3:2d} are inherited from properties
  \eqref{L:RS:2:a} and \eqref{L:RS:2:b} satisfied by $v^\e$.

  \step 2  Construction of $u^{\delta,\e}$.

  We set $u^{\delta,\e}:=\varphi^{\delta,\e}\circ u^{\delta,\e}_{m(\delta)}$, where
  $u^{\delta,\e}_{m(\delta)}\equiv u^{\e}_{m}$ is the
  isometry constructed in Step~1, and $\varphi^{\delta,\e}$ is a
  rigid motion, which is chosen in such a way that $u^{\delta,\e}$
  satisfies \eqref{ass:BC}. (Note that this is possible, since
  $u^{\delta,\e}_{m(\delta)}$ is affine on $L_{BC}$ by
  \eqref{L:RS:3:2a}). Recall the definition of $V_\delta$, see
  \eqref{eq:38}. From
  \eqref{L:RS:3:2a} -- \eqref{L:RS:3:2d} we learn that 
  \begin{align}
    \label{L:RS:3:3a}
    &\secf^{\delta,\e}=\secf\qquad\text{ on }S\setminus V_\delta,\\
    \intertext{and as $\e\downarrow 0$:}
    \label{L:RS:3:3b}
    &\secf^{\delta,\e}\wto\secf\qquad\text{weakly in }L^2(S)\text{ as
    }\e\downarrow 0,\\
    \label{L:RS:3:3c}
    &\lim\limits_{\e\downarrow 0}\int_{V_\delta}Q(\tfrac{x}{\e},\secf^{\delta,\e}(x))\,dx\\
    \notag
    &\qquad=
    \int_{V_\delta}(1-\chi_{\nabla u}(x))Q_{\av}(\secf(x))+\chi_{\nabla u}(x)Q_{\ho}(\secf(x))\,dx.
  \end{align}
  Since $u^{\delta,\e}$ satisfies \eqref{ass:BC}, we deduce from \eqref{L:RS:3:3b}
  that 
  \begin{equation}
\|u^{\delta,\e}-u\|_{L^2(S)}\to 0\quad\text{ as }
  \e\downarrow 0\,.\label{eq:39}
  \end{equation}
Next we estimate the difference
  $\mathcal E^{\e}(u^{\delta,\e})-\mathcal
  E^0(u)$. From \eqref{L:RS:3:3a} and \eqref{L:RS:3:3c} we deduce that
  \begin{eqnarray*}
    &&\lim\limits_{\e\downarrow 0}\mathcal
    E^{\e}(u^{\delta,\e})\\
    &=&\int_{V_\delta}(1-\chi_{\nabla
      u}(x))Q_{\av}(\secf(x))+\chi_{\nabla
      u}(x)Q_{\ho}(\secf(x))\,dx\\
    &&+\int_{S\setminus V_{\delta}}Q_{\av}(\secf(x))\,dx\\
    &=&\int_{S}(1-\chi_{\nabla
      u}(x))Q_{\av}(\secf(x))+\chi_{\nabla
      u}(x)Q_{\ho}(\secf(x))\,dx\\
    &&-\int_{S\setminus V_{\delta}}(1-\chi_{\nabla
      u}(x))Q_{\av}(\secf(x))+\chi_{\nabla
      u}(x)Q_{\ho}(\secf(x))\,dx\\
    &&+\int_{S\setminus V_{\delta}}Q_{\av}(\secf(x))\,dx    
  \end{eqnarray*}
  Since $Q_{\ho}(F)\leq Q_{\av}(F)\leq \frac{1}{\alpha}|\sym F|^2$, where
  $\alpha$ is the constant of ellipticity
  (cf. \eqref{ass:Q}), and because $S\setminus
  V_{\delta}\subset S\setminus E_\delta$, we finally get
  \begin{equation}
    \lim\limits_{\e\downarrow 0}\Big|\mathcal
    E^{\e}(u^{\delta,\e})-\mathcal
    E^0(u)\Big|\,\leq\,\frac{2}{\alpha}\int_{S\setminus E_\delta}|\secf(x)|^2\,dx    .\label{eq:35}
  \end{equation}
In combination with \eqref{eq:39}, this proves \eqref{eq:1} and thus completes the proof of the lemma.
\end{proof}

\bibliographystyle{alpha}

\bibliography{Neukamm-Olbermann-2013}   

\end{document}